\documentclass[11pt]{amsart}

\usepackage{amssymb}

\usepackage{graphics}

\usepackage{latexsym}

\usepackage{amsmath}

\usepackage{amsthm}

\usepackage{amssymb,amsthm,amsfonts}

\usepackage{amscd}

\usepackage[arrow, matrix, curve]{xy}

\usepackage{syntonly}

\usepackage{mathrsfs}

\usepackage{amssymb}

\usepackage{amsxtra}

\usepackage{graphics}

\usepackage{braket}

\usepackage{mathtools}

\usepackage[normalem]{ulem}

\usepackage{color}

\usepackage[square,sort,comma,numbers]{natbib}

\usepackage{tikz-cd}

\newcommand{\raju}[1]{\color{magenta}{ (Raju: #1) }\color{black}}

\title[ ]{Periodicity of Hitchin's uniformizing Higgs bundles}

\author[Raju Krishnamoorthy]{Raju Krishnamoorthy}

\author[Mao Sheng]{Mao Sheng}

\email{krishnamoorthy@uni-wuppertal.de}

\address{Department of Mathematics, Bergische Universit\"at Wuppertal, Fakult\"at Mathematik und Naturwissenschaften, Gau{\ss}  stra{\ss}e 20}

\email{msheng@ustc.edu.cn}

\address{School of Mathematical Sciences, University of Science and Technology of China, Hefei, 230026, China}

\begin{document}

	
	\theoremstyle{plain}
	
	\newtheorem{thm}{Theorem}[section]
	
	\newtheorem{theorem}[thm]{Theorem}
	\newtheorem*{theorem*}{Theorem}
	
	\newtheorem{lemma}[thm]{Lemma}
	
	\newtheorem{semisimple lemma}[thm]{Semisimplicity Lemma}
	
	\newtheorem{periodic lemma}[thm]{Periodicity Lemma}
	
	\newtheorem{corollary}[thm]{Corollary}
	
	\newtheorem{proposition}[thm]{Proposition}
	
	\newtheorem{observation}[thm]{Observation}
	
	\newtheorem{addendum}[thm]{Addendum}
	
	\newtheorem{variant}[thm]{Variant}
	
	
	\theoremstyle{definition}
	\newtheorem{setup}[thm]{Setup}
	
	\newtheorem{construction}[thm]{Construction}
	
	\newtheorem{perspective}[thm]{Perspective}
	
	\newtheorem{notations}[thm]{Notation}
	
	\newtheorem{notation}[thm]{Notation}
	
	\newtheorem{question}[thm]{Question}
	
	\newtheorem{problem}[thm]{Problem}
	
	\newtheorem{remark}[thm]{Remark}
	
	\newtheorem{remarks}[thm]{Remarks}
	
	\newtheorem{definition}[thm]{Definition}

	\newtheorem{claim}[thm]{Claim}
	
	\newtheorem{assumption}[thm]{Assumption}
	
	\newtheorem{assumptions}[thm]{Assumptions}
	
	\newtheorem{properties}[thm]{Properties}
	
	\newtheorem{example}[thm]{Example}
	
	\newtheorem{conjecture}[thm]{Conjecture}
	
	\numberwithin{equation}{thm}

	
	\newcommand{\pP}{{\mathfrak p}}
	
	\newcommand{\pH}{{\mathfrak h}}
	
	\newcommand{\pX}{{\mathfrak X}}
	
	\newcommand{\pY}{{\mathfrak Y}}
	
	\newcommand{\sA}{{\mathcal A}}
	
	\newcommand{\sB}{{\mathcal B}}
	
	\newcommand{\sC}{{\mathcal C}}
	
	\newcommand{\sD}{{\mathcal D}}
	
	\newcommand{\sE}{{\mathcal E}}
	
	\newcommand{\sF}{{\mathcal F}}
	
	\newcommand{\sG}{{\mathcal G}}
	
	\newcommand{\sH}{{\mathcal H}}
	
	\newcommand{\sI}{{\mathcal I}}
	
	\newcommand{\sJ}{{\mathcal J}}
	
	\newcommand{\sK}{{\mathcal K}}
	
	\newcommand{\sL}{{\mathcal L}}
	
	\newcommand{\sM}{{\mathcal M}}
	
	\newcommand{\sN}{{\mathcal N}}
	
	\newcommand{\sO}{{\mathcal O}}
	
	\newcommand{\sP}{{\mathcal P}}
	
	\newcommand{\sQ}{{\mathcal Q}}
	
	\newcommand{\sR}{{\mathcal R}}
	
	\newcommand{\sS}{{\mathcal S}}
	
	\newcommand{\sT}{{\mathcal T}}
	
	\newcommand{\sU}{{\mathcal U}}
	
	\newcommand{\sV}{{\mathcal V}}
	
	\newcommand{\sW}{{\mathcal W}}
	
	\newcommand{\sX}{{\mathcal X}}
	
	\newcommand{\sY}{{\mathcal Y}}
	
	\newcommand{\sZ}{{\mathcal Z}}

	\newcommand{\scrA}{{\mathscr A}}
	
	\newcommand{\scrD}{{\mathscr D}}
	
	\newcommand{\scrS}{{\mathscr S}}
	
	\newcommand{\scrU}{{\mathscr U}}
	
	\newcommand{\scrL}{{\mathscr L}}
	
	\newcommand{\scrM}{{\mathscr M}}
	
	\newcommand{\scrX}{{\mathscr X}}
	
	\newcommand{\scrE}{{\mathscr E}}
	
	\newcommand{\scrY}{{\mathscr Y}}
	
	\newcommand{\scrV}{{\mathscr V}}
	
	\newcommand{\scrF}{{\mathscr F}}
	
	
	\newcommand{\A}{{\mathbb A}}
	
	\newcommand{\B}{{\mathbb B}}
	
	\newcommand{\C}{{\mathbb C}}
	
	\newcommand{\D}{{\mathbb D}}
	
	\newcommand{\E}{{\mathbb E}}
	
	\newcommand{\F}{{\mathbb F}}
	
	\newcommand{\G}{{\mathbb G}}
	
	\newcommand{\HH}{{\mathbb H}}
	
	\newcommand{\I}{{\mathbb I}}
	
	\newcommand{\J}{{\mathbb J}}
	
	\renewcommand{\L}{{\mathbb L}}
	
	\newcommand{\K}{{\mathbb K}}
	
	\newcommand{\M}{{\mathbb M}}
	
	\newcommand{\N}{{\mathbb N}}
	
	\renewcommand{\P}{{\mathbb P}}
	
	\newcommand{\Q}{{\mathbb Q}}
	
	\newcommand{\R}{{\mathbb R}}
	
	\newcommand{\SSS}{{\mathbb S}}
	
	\newcommand{\T}{{\mathbb T}}
	
	\newcommand{\U}{{\mathbb U}}
	
	\newcommand{\V}{{\mathbb V}}
	
	\newcommand{\W}{{\mathbb W}}
	
	\newcommand{\X}{{\mathbb X}}
	
	\newcommand{\Y}{{\mathbb Y}}
	
	\newcommand{\Z}{{\mathbb Z}}
	
	\newcommand{\id}{{\rm id}}
	
	\newcommand{\rank}{{\rm rank}}
	
	\newcommand{\END}{{\mathbb E}{\rm nd}}
	
	\newcommand{\End}{{\rm End}}
	
	\newcommand{\Hom}{{\rm Hom}}
	
	\newcommand{\Hg}{{\rm Hg}}
	
	\newcommand{\tr}{{\rm tr}}
	
	\newcommand{\SL}{{\rm SL}}
	
	\newcommand{\PSL}{{\rm PSL}}
	
	\newcommand{\GL}{{\rm GL}}
	
	\newcommand{\Gr}{{\rm Gr}}
	
	\newcommand{\Cor}{{\rm Cor}}
	
	\newcommand{\HIG}{\mathrm{HIG}}
	
	\newcommand{\PHG}{\mathrm{PHG}}
	
	\newcommand{\MHG}{\mathrm{MHG}}
	
	\newcommand{\MIC}{\mathrm{MIC}}
	
	\newcommand{\DR}{\mathrm{DR}}
	
	\newcommand{\MDR}{\mathrm{MDR}}
	
	\newcommand{\PDR}{\mathrm{PDR}}
	
	\newcommand{\HDF}{\mathrm{HDF}}
	
	\newcommand{\DHF}{\mathrm{DHF}}

	\newcommand{\PaHG}{\mathrm{HG}^{\text{par}}}
	
	\newcommand{\PaDR}{\mathrm{DR}^{\text{par}}}
	
	\newcommand{\Spec}{\mathrm{Spec\ }}

	
	\thanks{This research work is supported by the National Key R and D Program of China 2020YFA0713100, CAS Project for Young Scientists in Basic Research Grant No. YSBR-032, National Natural Science Foundation of China (Grant No. 11721101), the Fundamental Research Funds for the Central Universities (No. WK3470000018) and Innovation Program for Quantum Science and Technology (2021ZD0302904).}

	\begin{abstract}
		We link the periodicity of Hitchin's uniformizing Higgs bundle with the arithmetic geometry of its underlying curve.  Some new relations are discovered. We also speculate on the whole class of periodic Higgs bundles. 
	\end{abstract}
	\maketitle
	\tableofcontents
	
	\section{Introduction}
	The uniformization theorem of compact Riemann surfaces of genus $g\geq 2$ is one of the most fundamental results in complex geometry. N. Hitchin \cite[Example 1.5]{H} proposed an alternative approach to this uniformization theory, via the so-called uniformizing Higgs bundle. Because of its intimate connection with geometry, the uniformizing Higgs bundle a la Hitchin is one of the most studied objects in the theory of Higgs bundles. In \cite{LS}, we studied the periodicity of the uniformizing Higgs bundle over $\P^1$ minus four points and its possible relation with the $\Z$-lattice structure of the associated uniformizing representation of $\pi_1$. In this article, we explore further into this connection for smooth \emph{projective} curves. 
	
	Let $C$ be a connected complex smooth projective curve of genus $\geq 1$. Choose and then fix a square root $K_C^{1/2}$ of the canonical bundle of $C$. Hitchin's uniformizing Higgs bundle $(E_{unif},\theta_{unif})$ attached to $C$ is given as follows:  $E_{unif}= K_{C}^{1/2}\oplus K_{C}^{-1/2}$ and $\theta_{unif}$ is given by the matrix:
	$$
	\theta_{unif}:= \left(\begin{matrix}0 & Id_{K_{C}^{1/2}}\\
		0 & 0
	\end{matrix}\right)\colon E_{unif}\rightarrow E_{unif}\otimes K_{C}= K_{C}^{3/2}\oplus K_{C}^{1/2}.
	$$
	The uniformizing Higgs bundle is uniquely defined up to tensoring with a two-torsion line bundle.  A \emph{spreading-out} of $(C,E_{unif},\theta_{unif})$ is a triple $(\sC, \sE_{unif},\Theta_{unif})$ defined over $S$, where $S=\Spec(A)$ with $A\subset \C$ a finitely generated $\Z$-subalgebra, $\sC$ a smooth projective curve over $S$, $(\sE_{unif},\Theta_{unif})$ an $S$-relative Higgs bundle over $\sC$, together with an isomorphism over $\C$:
	$$\alpha:(\sC, \sE_{unif},\Theta_{unif})\times_A\C\cong (C,E_{unif},\theta_{unif}).$$
	\begin{definition}
		The Higgs bundle $(E_{unif},\theta_{unif})$ over $C$ is said to be \emph{periodic} if there exists
		\begin{itemize}
		\item a spreading-out 
		$
		(\sC, \sE_{unif},\Theta_{unif})\to S
		$
		of $(C,E_{unif},\theta_{unif})$
		\item a proper closed subscheme $Z\subset S$;
		\item and a positive integer $f$
		\end{itemize}
		 such that the reduction $(\sE_{unif,s},\Theta_{unif,s})$ at \emph{every} geometric point of $s\in S-Z$ is periodic of period $\leq f$ with respect to \emph{any} $W_2(k(s))$-lifting $\tilde s\to S$ of $s$. It is said to be \emph{one-periodic}, if $f$ can be taken to be one. 
	\end{definition}
	It follows from \cite[Lemma 4.8]{KS} that if $(E_{unif},\theta_{unif})$ is periodic with respect to one spreading-out of $C$, it is periodic with respect to any spreading-out. Also, $(E_{unif},\theta_{unif})$ is (one-)periodic if and only if $(E_{unif},\theta_{unif})\otimes (L,0)$ is (one-)periodic for any (two-)torsion line bundle $L$. Therefore, Hitchin's uniformizing Higgs bundle being (one-)periodic/non (one-)periodic is an \emph{intrinsic} property of its underlying curve. Our findings may be summarized as follows.
	\begin{definition}\label{infty supersingular reduction}
		A complex elliptic curve $C$ has \emph{infinitely many primes of supersingular reduction} if for a spreading-out $\mathcal C\rightarrow S$ of $C$, there exist geometric points $s\in S$ of arbitrarily large residue characteristic such that the reduction $\mathcal C_s$ is supersingular.
	\end{definition}
	Clearly, whether a complex elliptic curve has infinitely many primes of supersingular reduction or not is independent of the choice of spreading-out.
	\begin{proposition}[Proposition \ref{nonperiodicity for unif}]\label{nonperiodicity of elliptic curve}
		If a complex elliptic curve $C$ has infinitely many primes of supersingular reduction, then $(E_{unif},\theta_{unif})$ over $C$ is non-periodic. 
	\end{proposition}
	It would follow from the polystability conjecture (see Conjecture \ref{polystability} below), that $(E_{unif},\theta_{unif})$ over a complex elliptic curve is non-periodic without any condition on primes of supersingular reduction. On the other hand, the non-periodicity of $(E_{unif},\theta_{unif})$ also follows from a conjecture of Serre. Indeed, if the $j$-invariant of $C$ is transcendental, then $C$ has infinitely many primes of supersingular reduction \footnote{To see this, one may take the Legendre family $\{Y^2=X(X-Z)(X-tZ)\}\subset \P^2_S$ over $S=\Spec(\Z[t,\frac{1}{t},\frac{1}{t-1}])$ as a spreading-out of $C$.}. The condition in Proposition \ref{nonperiodicity of elliptic curve} for $C$ defined over \emph{number field} is exactly conjectured by Serre. This conjecture has been proved by Elkies, for $C$ defined either over a number field of \emph{odd} degree over $\Q$ \cite{Elkies} or a number field that admits at least one real embedding \cite{Elkies89}. Surprisingly, the polystability conjecture in the case of elliptic curve is equivalent to the conjecture of Serre.
	
	\begin{theorem}[Theorem \ref{classification}]
		Let $C$ be a complex elliptic curve. Then the following two statements are equivalent.
		\begin{itemize}
			\item [(i)] $C$ has infinitely many primes of supersingular reduction.
			\item [(ii)] The category of periodic Higgs bundles over $C$ is semisimple whose simple objects are torsion line bundles equipped with the zero Higgs field.
		\end{itemize}
	\end{theorem}
	Property (ii) is predicted by the aforementioned polystability conjecture, because polystable Higgs bundles over elliptic curves must be direct sums of line bundles, and we show that, over any smooth projective curve, a direct sum of line bundles is periodic iff each factor is torsion (Corollary \ref{direct sum of line bundles}). (This latter fact is a Higgs analog of the Grothendieck-Katz $p$-curvature conjecture for semi-simple abelian flat connections, shown by Chudnovsky-Chudnovsky.)

	We turn to higher genus curves. Note that $(E_{unif},\theta_{unif})$ over a hyperbolic curve $C$ is stable. Hitchin's uniformization theory is encoded in the following diagram:
	
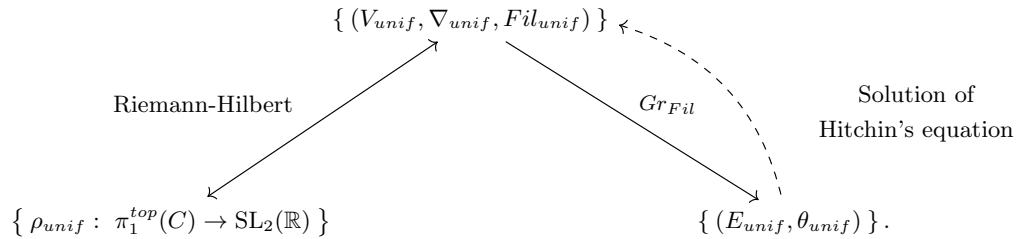
\vspace{3mm}
	\hspace*{-3em} 
	\begin{tikzpicture}[commutative diagrams/every diagram]
		\matrix[matrix of nodes, name=a] {
			&
			& [-3em] |[name=a12]| {\footnotesize $\Set{(V_{unif},\nabla_{unif}, Fil_{unif})}$}
			& \\
			\\
			\\
			& |[name=a21]| {\footnotesize $\Set{\rho_{unif}: \ \pi^{top}_1(C)\to \mathrm{SL}_2(\mathbb{R})}$}
			&
			& [-2em] |[name=a23]| {\footnotesize $\Set{(E_{unif},\theta_{unif})}.$}
			\\
		};
		\path[commutative diagrams/.cd, every arrow, every label]
		(a12) edge[commutative diagrams/rightarrow] node[auto]{\(Gr_{Fil}\)} (a23)
		(a12) edge[commutative diagrams/leftarrow, dashed, bend left] node[auto, pos=0.82]{\(\parbox{9.8em}{\begin{center}{\footnotesize Solution of \\ Hitchin's equation }\end{center}}\)} (a23)
		(a21) edge[commutative diagrams/leftrightarrow] node[auto]{\(\parbox{7em}{\footnotesize{Riemann-Hilbert}}\)} (a12)
		;\end{tikzpicture}

	\vspace{3mm}
	
	In the diagram, $V_{unif}$ is the unique nontrivial extension of $K_C^{-\frac{1}{2}}$ by $K_C^{\frac{1}{2}}$ and $Fil_{unif}$ is the unique sub line bundle of maximal degree. Faltings \cite{Fa83} searched for an algebraic construction of $\nabla_{unif}$ in the diagram, but came up with a negative result. Our study of periodicity of $(E_{unif},\theta_{unif})$ provides a different perspective on this negative result.
	
	We call a smooth projective curve $C/\mathbb{C}$ of genus $g\geq 2$ \emph{generic} if the associated moduli map $\Spec(\C)\rightarrow \sM_g$ to the moduli stack of genus $g$ curves over $\Spec(\mathbb{Z})$ is dominant. 
	\begin{proposition}[Proposition \ref{non-periodicity}] Let $C$ be a generic hyperbolic curve of genus $g$. Then $(E_{unif},\theta_{unif})$ is non one-periodic. Consequently, away from a countable union of proper closed subsets of $\sM_g(\C)$, $(E_{unif},\theta_{unif})$ is non one-periodic.
	\end{proposition}
	Note that if $(E_{unif},\theta_{unif})$ over a generic hyperbolic curve was one-periodic, then $\nabla_{unif}$ could be reconstructed from the inverse Cartier transforms of the various mod $p$ reductions of $(E_{unif},\theta_{unif})$. One may even speculate that, over a generic hyperbolic curve, $(E_{unif},\theta_{unif})$ is non-periodic. We indeed conjecture this to be the case, but proving this seems substantially more difficult than the one-periodic case. 
	
	Let $U$ be a nonempty Zariski open subset of $C$ and $f: X\to U$ a smooth projective morphism. For each $i\geq 0$, one associates to $f$ the $i$-th \emph{parabolic Gau\ss-Manin system}, which is a parabolic de Rham bundle defined over $C$ (see the paragraph following Lemma 3.3 \cite{KS}). When the $i$-th higher direct image $R^if_*\C_X$ has trivial local monodromies at infinity, it is actually a de Rham bundle over $C$. 
	\begin{definition}\label{motivic}
		Let $C$ be a smooth projective curve over $\C$ and $(E,\theta)$ a graded Higgs bundle over $C$. We say $(E,\theta)$ is \emph{motivic}, if there is a smooth projective morphism $f: X\to U$, where $U$ is some nonempty Zariski open subset of $C$, and some $i\geq 0$, such that $(E,\theta)$ is isomorphic to a direct factor of the associated parabolic graded Higgs bundle to the $i$-th parabolic Gau\ss-Manin system. 
	\end{definition}
	In the above definition, we regard $(E,\theta)$ over $C$ as a parabolic graded Higgs bundle with empty parabolic support (see \cite[Definition 3.4]{KS} for the general case). The Higgs periodicity theorem (\cite[Theorem 1.3]{LS}, \cite[Theorem 5.6]{KS}) asserts that \emph{a motivic Higgs bundle is periodic}.  
	\begin{theorem}[Theorem \ref{intro_geo_modular}]
		Let $C$ be a smooth projective hyperbolic curve. Then the following statements are equivalent.
		\begin{enumerate}
			\item [(i)] $(E_{unif},\theta_{unif})$ is motivic.
			\item [(ii)] $C$ admits a modular embedding.
		\end{enumerate}
	\end{theorem}
	\begin{remark}\label{remark on modular embedding}
		The theory of modular embeddings is due primarily to Cohen-Wolfart, Schmutz Schaller-Wolfart, and Kucharczyk \cite{Wolf,SM,Kuc}. A conjecture of Chudnovsky-Chudnovsky \cite{CC90} asserts that curves admiting modular embeddings form a narrow class: Shimura curves and triangle curves (see \cite[Problem 1, Remark 9]{SM}). Our periodic Higgs conjecture \ref{motivicity conj} predicts that $(E_{unif},\theta_{unif})$ being periodic actually means motivicity. These two conjectures combine into a satisfactory answer to the periodicity problem of $(E_{unif},\theta_{unif})$. However, both conjectures remain largely open so far. 
	\end{remark}
	
	As remarked, the uniformizing Higgs bundle attached to a Shimura curve is periodic. The theory of canonical models of Shimura varieties provides a \emph{distinguished} spreading-out for a Shimura curve. Periodicity of $(E_{unif},\theta_{unif})$ with respect to this spreading-out yields a Hasse-Witt invariant for each good reduction of a Shimura curve. Let $F$ be a totally real field and $D$ a quaternion algebra over $F$ which splits at exactly one real place $\tau$. Let $G_{\Q}$ be the reductive $\Q$-group $Res_{F/\Q}D^\times$. Let $M$ be an associated Shimura curve of Hodge type defined by a symplectic representation $G_{\Q}\to GSp_{2n}$ together with a level structure. Let $\mathfrak p$ be an odd prime of $F$ such that $\frak p$ is unramified over $\Q$ and does not divide the discriminant of $D$. Set $\F_{q}=k_{\mathfrak p}$, the residue field of $F$ at $\mathfrak p$. By a result of Kisin \cite[Theorem 2.3.8]{Ki}, there is a canonical smooth integral model $\scrM$ of $M$ over the ring of integers $\sO_{F_{\mathfrak{p}}}$, together with an abelian scheme over the integral model. The Hasse-Witt invariant, derived from the periodicity of $(E_{unif},\theta_{unif})$ at $\mathfrak p$, enables us to resolve Conjecture 1.3 \cite{SZZ}, and we obtain thereby a \emph{Deuring-Eichler mass formula} for a Shimura curve of Hodge type.
	\begin{theorem}[Corollary \ref{independence}, Corollary \ref{massformula}]\label{main_result}
		Let $\sM\hookrightarrow \sA_n$ be the good reduction of a Shimura curve as above.
		\begin{itemize}
			
			\item [(i)] There are exactly two Newton strata appearing in $\sM(\overline \F_p)$;
			
			\item [(ii)] The Newton jumping locus $\sN_p\subset \sM(\overline \F_p)$ is independent of symplectic representation of $G_{\Q}$;
			
			\item [(iii)] The mass formula reads:
			
			$$|\sN_p|=(1-q)\frac{\chi_{top}(M(\C))}{2}.$$
			
		\end{itemize}
		
	\end{theorem}
	\begin{remark}
		There has been much recent work on the Ekedahl-Oort strata on the good reduction of Shimura varieties of Hodge type, in particular by Goldring-Koskivirta and Zhang \cite{Gold,Zhang}. Their work yields another approach to parts of Theorem \ref{main_result}, but the techniques are rather different. We briefly discuss this circle of ideas.
		
		Many generalized Hasse invariants for the good reduction of Hodge-type Shimura varieties have been recently constructed in \cite{Gold} by constructing Hasse invariants on the stack of $G$-zips. These Hasse invariants cut out the (reduced) Ekedahl-Oort strata and are ``automatically" Hecke-equivariant. The stack of $G$-zips is analogous to a period domain. Call their Hasse invariant $H^{GK}$. Zhang proves that the open Ekedahl-Oort stratum is independent of the choice of symplectic embedding realizing $(G,X)$ as a Shimura datum of Hodge-type \cite[Theorem 0.1]{Zhang}, which implies Theorem \ref{main_result} (ii). Hence the dependence of $H^{GK}$ on the choice of symplectic embedding is minimal and in particular the support of $\mathrm{div}(H^{GK})$ is independent of the symplectic embedding. On the other hand, it is not clear that $\mathrm{div}(H^{GK})$ is multiplicity-free.  
	\end{remark}

	We conclude the introduction with the following 
	\begin{conjecture}[Polystability conjecture]\label{polystability}
		Any periodic Higgs bundle over a smooth projective complex curve $C$ is polystable.
	\end{conjecture} 
	Note that this conjecture is a consequence of our (much more ambitious) periodic Higgs conjecture, claiming that periodicity implies motivicity. As a matter of convention, a \emph{Hodge} filtration on a flat connection over $C$ means an effective finite decreasing Griffiths transverse filtration of subbundles whose associated graded Higgs bundle is semistable.

	\section{Elliptic curves}
	Let $C$ be a complex elliptic curve. In this section, we choose $K_C^{1/2}=\sO_C$ for convenience. A spreading-out $\sC\to S$ of $C$ yields a spreading-out $(\sO_{\sC}^{\oplus 2}, id)$ of $(E_{unif},\theta_{unif})$ over $S$. For a geometric point $s\in S$, we use the subscript $s$ to denote for the reduction of an object at $s$.

	\subsection{Non-periodicity}

	\begin{lemma}\label{criterion for uniformizing}
		Notation as above. $(\sO_{\sC_s}^{\oplus 2}, id)$ is non periodic at $s$ if and only if the bundle of $C^{-1}(\sO_{\sC_s}^{\oplus 2}, id)$ is isomorphic to $\sN_s$, the unique nontrivial extension of $\sO_{\sC_s}$ by $\sO_{\sC_s}$.
	\end{lemma}
	\begin{proof}
		Let $V$ be the bundle of $C^{-1}(\sO_{\sC_s}^{\oplus 2}, id)$. Suppose that $V$ is isomorphic to $\sN_s$, the unique non-split extension of $\sO_{\sC_s}$ by $\sO_{\sC_s}$. Then possible Hodge filtrations are either $Fil=\sN_s$ or $Fil=\sO_{\sC_s}$, the unique sub line bundle of $\sN_s$ of degree zero. Thus, $\Gr \circ C^{-1}(\sO_{\sC_s}^{\oplus 2}, id)$ is isomorphic to either $(\sN_s,0)$ or $(\sO_{\sC_s}^{\oplus 2},0)$. As the Higgs field is zero in either case, the bundle of its inverse Cartier transform is simply its Frobenius pullback and the connection is the canonical connection. Therefore, the associated graded of its inverse Cartier transform with respect to \emph{any} Hodge filtration can only be isomorphic to one of these two possibilities. Consequently, it never flows back to $(\sO_{\sC_s}^{\oplus 2}, id)$. Thus, $C^{-1}(\sO_{\sC_s}^{\oplus 2}, id)$ is non periodic. Conversely, suppose the contrary, i.e., that $V$ is non-isomorphic to $\sN_s$. As $V$ is always an extension of $\sO_{\sC_s}$ by $\sO_{\sC_s}$ and there are only two such isomorphism classes of extensions, $V$ must be isomorphic to $\sO_{C_s}^{\oplus 2}$. Since the connection of $C^{-1}(\sO_{\sC_s}^{\oplus 2}, id)$ has nonzero $p$-curvature, there exists a sub line bundle which is isomorphic to $\sO_{C_s}$ and not invariant under the connection. Taking it as the Hodge filtration of $C^{-1}(\sO_{\sC_s}^{\oplus 2}, id)$, one obtains the one-periodicity for $(\sO_{C_s}^{\oplus 2}, id)$ in this case, which is a contradiction. The lemma is proved.
	\end{proof}
	Let $s\in S$ be a geometric point with residue field $k$, and let $\tilde s: W_2(k)\to S$ be a first-order thickening of $s$ inside $S$. Let $\sC_{\tilde s}$ be the restriction of $\sC$ over $\tilde s$, which is a $W_2(k)$-lifting of $\sC_s$. We write $C^{-1}_{\sC_s\subset \sC_{\tilde s}}(\sO_{\sC_s}^{\oplus 2}, id)$ for the inverse Cartier transform of $(\sO_{\sC_s}^{\oplus 2}, id)$, to emphasize its dependence on the choice of a $W_2$-lifting.  
	\begin{lemma}\label{structure of V}
		The bundle of $C^{-1}_{\sC_s\subset \sC_{\tilde s}}(\sO_{\sC_s}^{\oplus 2}, id)$ is isomorphic to $\sN_s$ if and only if the obstruction class of lifting the absolute Frobenius $F_{\sC_s}$ over $\sC_{\tilde s}$ is nonzero. 
	\end{lemma}
	\begin{proof}
		We use the reinterpretation of the inverse Cartier transform via exponential twisting \cite{LSZ0}. Let $\theta_{unif}: L_1\to L_2\otimes K_{\sC_s}$ be the uniformizing Higgs field, where $L_1\cong L_2\cong K_{\sC_s}\cong \sO_{\sC_s}$, and let $\{h_{\alpha\beta}\}$ be a \v{C}ech representative of the obstruction class $ob(F)$ of lifting $F:=F_{\sC_s}$ (see \cite[\S2.1] {LSZ0}). Then by its very construction (see \cite[\S2.2]{LSZ0}), $V$ is the extension of $F^*L_1$ by $F^*L_2$, and the extension class has a \v{C}ech representative $h_{\alpha\beta}(F^*\theta_{unif})$. In other words, the extension class is the image of $ob(F)\otimes F^*\theta_{unif}$ under the natural morphism induced by $F^*T_{\sC_s}\otimes F^*K_{\sC_s}\to \sO_{\sC_s}$:$$H^1(\sC_s,F^*T_{\sC_s})\otimes\Hom(F^*L_1,F^*L_2\otimes F^*K_{\sC_s})\to H^1(F^*L_1,F^*L_2).$$ As $F^*\theta_{unif}$ is an isomorphism, it is (non)zero if and only if $ob(f)$ is (non)zero. The lemma follows.
	\end{proof}
	Notice the change of noations in the next lemma.
	\begin{lemma}\label{obstruction of Frob for supersingular}
		Let $k$ be an algebraically closed field of characteristic $p>0$. Let $C$ be a supersingular elliptic curve over $k$ and $\tilde C$ be a $W_2(k)$-lifting. Then the obstruction of lifting the absolute Frobenius $F_{C}$ of $C$ over $\tilde C$ is nonzero.
	\end{lemma}
	\begin{proof}
		Assume there is a $W_2$-morphism $\tilde F: \tilde C\to \tilde C$ lifting $F_C$. Let $0$ be the identity of $\tilde C$. Then $\tilde F(0)$ modulo $p$ is the identity of $C$. So $\tilde F':=\tilde F-\tilde F(0)$ is another lifting of $F$, which is a group homomorphism by \cite[Corollary 6.4]{MFK}. It follows that $\ker(\tilde F')$ is a $W_2$-lifting of the group scheme $\alpha_p=\ker(F_C)$, which is however impossible.   
	\end{proof}
	\begin{proposition}\label{nonperiodicity for unif}
		Let $C$ be a complex elliptic curve. Suppose that $C$ has infinitely many primes of supersingular reduction, then $(\sO_C^{\oplus 2},id)$ is non periodic.  
	\end{proposition} 
	\begin{proof}
		Combine Lemmata \ref{criterion for uniformizing}-\ref{obstruction of Frob for supersingular}.
	\end{proof}
	
	\subsection{Periodic Higgs bundles over elliptic curves}
	We start with the classification of rank one periodic Higgs bundles over an \emph{arbitrary} complex smooth projective curve.
	\begin{proposition}\label{periodic_line_bundles}
		Let $C$ be a complex smooth projective curve and $(L,\theta)$ a Higgs line bundle over $C$. Then $(L,\theta)$ is periodic if and only if $\theta=0$ and $L$ is torsion. 
	\end{proposition}
	\begin{proof}
		For a line bundle $L$, we take a spreading-out $(\sC,\sL)$ of $(C,L)$ over $S$. Suppose first $L$ to be torsion, say of order $N\in \N$. Shrinking $S$ if necessary, we may assume the characteristics of residue fields of geometric points of $S$ are coprime to $N$. Then for an $s\in S(k)$ of characteristic $p$, one has 
		$$
		\sL_s^{\otimes p^{\phi(N)}}\cong \sL_s.
		$$   
		Therefore, $(L,0)$ is periodic.
		
		As a periodic Higgs bundle has nilpotent Higgs field, the Higgs field of a rank one periodic Higgs bundle must vanish. Shrinking $S$ if necessary, we may assume $(\sL_s,0)$ is periodic for any geometric point $s\in S$, i.e., there exists an $M\in \N$, independent of $s$, such that $\sL_s^{\otimes p^{M}-1}$ is trivial. Regard $[L]$ as an element in $\mathrm{Pic}^0(\sC/S)(k(s))$, i.e., a $k(s)$ point of the relative Picard scheme $\mathrm{Pic}^0(\sC/S)$. Thanks to a deep result of Pink \cite[Theorem 5.3]{Pink}, it follows that, for any $v\in S(\bar \Q)$, the specialization of $[L]$ at $v$, which is an element in $\mathrm{Pic}^0(\sC/S)(k(v))$, is torsion, Finally, applying the main theorem of \cite{Ma89} due to Masser, $[L]\in \mathrm{Pic}^0(\sC/S)(k(s))$ is torsion. In particular, $L$ is torsion over $C$. This concludes the proof.
	\end{proof}
	\begin{corollary}\label{direct sum of line bundles}
		Let $C$ be a complex smooth projective curve. Then a direct sum of line bundles $(\oplus_iL_i,0)$ is periodic if and only if each $L_i$ is torsion. 
	\end{corollary}
	\begin{proof}
		By Proposition \ref{periodic_line_bundles}, it suffices to show $(\oplus_iL_i,0)$ is periodic iff each $(L_i,0)$ is torsion. The if-direction is obvioius. For the only if-direction, we do induction on the number of direct factors. Claim that $(L_1,0)$ is periodic. Take a spreading-out $(\sC,\oplus_i\sL_i)$ over $S$, such that for any geometric point $s\in S$, $(\oplus_i\sL_{i,s},0)$ is periodic of period $f$. As $(\sL_{1,s},0)\subset (\oplus_i\sL_{i,s},0)$ is a Higgs subbundle, $(F^*_{\sC_s}\sL_{1,s},\nabla_{can})$ is a flat subbundle of $C^{-1}(\oplus_i\sL_{i,s},0)$. It follows that $(F^*_{\sC_s}\sL_{1,s},0)$ is a Higgs subbundle of $\Gr\circ C^{-1}(\oplus_i\sL_{i,s},0)$. Iterating the argument, we obtain the inclusion of Higgs subbundles
		$$
		\iota: ((F^*_{\sC_s})^{f}\sL_{1,s},0)\hookrightarrow (\Gr\circ C^{-1})^f(\oplus_i\sL_{i,s},0).
		$$
		By periodicity, one has an isomorphism of graded Higgs bundles $$\phi: (\Gr\circ C^{-1})^f(\oplus_i\sL_{i,s},0)\cong (\oplus_i\sL_{i,s},0).$$ Consider the composite of $\phi\circ \iota$ with various projections $\oplus_i\sL_{i,s}\to \sL_{i,s}$. Since $\deg(\sL_{i,s})$s are all zero, it is either zero or an isomorphism. It cannot be zero for all $i$s. So $(\sL_{1,s},0)$ is indeed perodic with period $f$. The claim is proved. As the category of periodic Higgs bundles over $C$ is a rigid tensor category (\cite[Proposition 6.2]{KS}), the cokernel of the natural inclusion $(L_1,0)\subset (\oplus_iL_i,0)$ of periodic Higgs bundles is again periodic. The induction hypothesis concludes the proof. 
	\end{proof}
	
	According to the periodic Higgs conjecture, classifying periodic Higgs bundles is of clear geometric significance. In this subsection, we consider the classification problem in the case of elliptic curves. 
	
	\begin{theorem}\label{classification}
		Let $C$ be a elliptic curve over $\C$. Then the following two statements are equivalent:
		\begin{itemize}
			\item [(i)] $C$ has infinitely many primes of supersingular reduction.
			\item [(ii)] The category of periodic Higgs bundles over $C$ is semisimple whose simple objects are torsion line bundles equipped with the zero Higgs field.
		\end{itemize}
	\end{theorem}
	Let $N$ be the unique (up to isomorphism) nontrivial extension of $\sO_C$ by $\sO_C$. 
	\begin{lemma}\label{The case of N}
		The Higgs bundle $(N,0)$ is non periodic if and only if $C$ has infinitely many primes of supersingular reduction.
	\end{lemma} 
	\begin{proof}
		Take a spreading-out $\sC/S$ of $C$. Let $\sN$ be the extension of $\sO_{\sC}$ by $\sO_{\sC}$ given by $1\in \mathrm{Ext}^1(\sO_{\sC},\sO_{\sC})$. Let $s\in S$ be a geometric point. We claim that $(\sN_s,0)$ is non periodic iff $\sC_s$ is supersingular: Suppose $\sC_s$ to be supersingular. Then
		$$
		C^{-1}(\sN_s,0)\cong (\sO_{\sC_s}^{\oplus 2},\nabla_{can}).
		$$
		It follows that, with respect to \emph{any} Hodge filtration over $C^{-1}(\sN_s,0)$, $\Gr\circ C^{-1}(\sN_s,0)\cong (\sO_{\sC_s}^{\oplus 2},0)$. However, under flow operator (with respect to any Hodge filtration), the trivial Higgs bundle goes only to the trivial Higgs bundle. Therefore, $(\sN_s,0)$ is non periodic. On the other hand, suppose now $\sC_s$ to be ordinary. Then
		$$
		C^{-1}(\sN_s,0)\cong (\sN_s,\nabla_{can}).
		$$
		Thus $\Gr_{Fil_{tr}}\circ C^{-1}(\sN_s,0)\cong (\sN_s,0)$, where $Fil_{tr}$ stands for the trivial Hodge filtration. So $(\sN_s,0)$ is one-periodic when $\sC_s$ is ordinary, and the claim is proved. The lemma follows. 
	\end{proof}
	
	The above lemma settles one direction of Theorem \ref{classification}. To handle with the classification problem, we start with a classical result of Atiyah (\cite[Theorem 9]{Ati}) on vector bundles over elliptic curves.
	\begin{lemma}[Atiyah]\label{structure of indec bundle}
		An indecomposable vector bundle of rank $r$ over $C$ is of form $S^{r-1}N\otimes L$, where $L$ is a line bundle.
	\end{lemma}

	A Higgs bundle is \emph{indecomposable} if it cannot be written into direct sum of two proper Higgs subbundles. A periodic Higgs bundle must be graded and semistable of degree zero.
	\begin{lemma}\label{structure of indec graded Higgs bundle}
		An indecomposable semistable graded Higgs bundle $(E,\theta)$ of degree zero over $C$ is, up to tensoring with a line bundle of degree zero, a successive extension of $(\sO_C,0)$s.
	\end{lemma}
	\begin{proof}
		Up to a shift of indices, we may write $E=\bigoplus_{0\leq i\leq n} E^i$ with $\theta: E^i\to E^{i-1}$ nonzero for each $i>0$. By Lemma \ref{structure of indec bundle}, $E^i$ is written into a direct sum
		$$
		E^i=\oplus_j L^{ij}\otimes S^{r_{ij}}N,
		$$
		where $L^{ij}$ is some line bundle. We claim that the $L^{ij}$s are all isomorphic to each other. We first observe the following simple fact: If there is a nonzero morphism $L\otimes S^rN\to L'\otimes S^{r'}N$, then $\deg L\leq  \deg L'$; furthermore, if $\deg L=\deg L'$, then $L\cong L'$.  Note also that, if $\theta(L^{ij}\otimes S^{r_{ij}}N)=0$ for some $ij$, then $\deg L^{ij}\leq 0$ by semistability. As $\theta$ is nilpotent, it follows that $\deg L^{ij}\leq 0$ for each $ij$. But since $\deg E=0$, one must have $\deg L^{ij}=0$ for each $ij$. Assume $E^n$ contains $m$ indecomposable factors. We denote by $S_{il}, 1\leq l\leq m$ the subset of $L^{ij}$s in $E^i$ with $L^{ij}\cong L^{nl}$. First, each $L^{ij}$ belongs to some $S_{il}$. This is because the direct sum of those indecomposable factors in $E^i$ wth $L^{ij}$ in some $S_{il}$, together with the induced Higgs field, forms a direct factor of $(E,\theta)$ which must be the whole $(E,\theta)$ by indecomposability. Second, for any $l_1\neq l_2$, there exists some $i$ such that $S_{il_1}\cap S_{il_2}\neq \emptyset$. This is also clear, for otherwise the direct sum of those indecomposable factors in $E^i$ wth $L^{ij}$ in $S_{il}$ would form a proper direct factor of $(E,\theta)$. Therefore, $L^{nl}$s are isomorphic to each other by the second point, and all $L^{ij}$s are isomorphic to each other by the first point. The claim is proved. Therefore,
		$$
		(E,\theta)=(L,0)\otimes (E',\theta'),
		$$
		where $L$ is a line bundle of degree zero and each indecomposable factor of $E'$ is of form $S^rN$. To show that $(E',\theta')$ is a successive extension of $(\sO_C,0)$s, we do induction on rank of $E'$. Note that we shall not use the indecomposibility of $(E',\theta')$ in the following argument. So $E'=\oplus_{0\leq i\leq n}E^{'i}$ with $E^{'i}=\oplus_jS^{r_{ij}}N$. Any $S^{r_{nj}}N$ contains a unique subline bundle $\sO_C$ which gives rise to a Higgs sub line bundle $(\sO_C,0)$ of $(E',\theta')$. Note that $(E',\theta')/(\sO_C,0)$ is again graded and semistable of degree zero, and its bundle part is in fact a direct sum of $S^rN$s. By induction, it is a successive extension of $(\sO_C,0)$s. The lemma follows.
	\end{proof}
	\begin{lemma}\label{maximal element}
		Let $(E,\theta)$ be a semistable graded Higgs bundle of degree zero over $C$. Let $\Sigma$ be the set of direct sums of degree zero Higgs sub line bundles of $(E,\theta)$. Then $\Sigma$ is nonempty and has a unique maximal element.
	\end{lemma}
	\begin{proof}
		We write $(E,\theta)=\oplus_i(E_i,\theta_i)$ into a direct sum of indecomposable graded Higgs bundles. Clearly, each $(E_i,\theta_i)$ is semistable of degree zero. By Lemma \ref{structure of indec graded Higgs bundle}, each $(E_i,\theta_i)$ contains a degree zero Higgs sub line bundle. Hence $S$ is nonempty. Assume we have two maximal elements in $\Sigma$, say $L=\oplus_i L_i$ and $M=\oplus_jM_j$. Consider the summation map
		$$
		\alpha: (L\oplus M,0) \to (E,\theta).
		$$
		As both sides of $\alpha$ are semistable of the same degree, $\mathrm{im}(\alpha)$ is also semistable of degree zero. As $L\oplus M$ is polystable of degree zero, $\mathrm{im}(\alpha)$ is a direct sum of degree zero line bundles and has strictly larger rank than that of $L$, contradicting the maximality of $L$. The lemma is proved.
	\end{proof}
	Now we proceed to the proof of Theorem \ref{classification}.
	\begin{proof} 
		It remains to prove that, if $C$ has infinitely many primes of supersingular reduction, then (ii) in the theorem holds. 
		
		We conisder first a class of graded semistable Higgs bundles over $C$, which are \emph{nonzero} extensions of $(\sO^{\oplus r}_C,0)$ by $(\sO_C^{\oplus s},0)$ for $r,s\geq 1$. We shall prove that they are non periodic. Let $E=(E,\theta)$ be such a Higgs bundle. We do induction on rank of $E$. Proposition \ref{nonperiodicity for unif} and Lemma \ref{The case of N} together imply the rank two case, namely $r=s=1$. In general, by Lemma \ref{structure of indec graded Higgs bundle}, the bundle of $E$ is a direct sum of $S^rN$s, and there exists a Higgs subbundle $(\sO_C,0)$ of $E$, giving rise to a short exact sequence of graded Higgs bundles:
		$$
		0\to (\sO_C,0)\to E\stackrel{\pi}{\to} F\to 0.
		$$
		If $F$ is a nonzero extension of $(\sO^{\oplus r'},0)$ by $(\sO^{\oplus s'},0)$ for some $r', s'$, then $F$ is non periodic by inducion hypothesis. As $(\sO_C,0)$ is certainly periodic, $E$ cannot be periodic by \cite[Proposition 6.2]{KS}. So we may assume $F$ is the trivial Higgs bundle $(\sO_C^{\oplus n},0)$ and the above exact sequence is non-split. We consider various projections 
		$$
		p_i: F= \sO_C^{\oplus n}\to \sO_C^{\oplus n-1}, \quad 0\leq i\leq n,
		$$
		where the $i$-th projection means omitting the $i$-th factor of $F$. Set $E_i$ to be the kernel of $\pi\circ p_i$. {We claim} that at least one of $E_i$, which is a rank two graded Higgs bundle, must be a nontrivial extension of $\sO_C$ by $\sO_C$. The claim furnishes the proof for the induction step because of the rank two case and \cite[Proposition 6.2]{KS}. To show the claim, we assume the contrary that each $E_i$ is trivial and look at the summation map
		$$
		\beta: \oplus_{i}E_i\to E
		$$
		It is easy to check the following properties of $\beta$: i) $\ker(\beta)$ is isomorphic to $\sO_C^{\oplus n}$; ii) the composite $\pi\circ \beta: \oplus_iE_i\to F$ is surjective. Hence $\mathrm{im}(\beta)=\sum_iE_i\subset E$ gives rise to a splitting of $\pi$. This is a contradiction.
		
		Let $E=(E,\theta)$ be a periodic Higgs bundle over $C$. We prove (ii) by induction on the rank of $E$. The rank one case follows from Proposition \ref{periodic_line_bundles}. Let $F=\oplus_{i\in \Sigma}L_i\subset E$ be the direct sum of degree zero Higgs sub line bundles, where $\Sigma$ is givn in Lemma \ref{maximal element}. Take a spreading-out $(\sC,\sE,\sF)/S$ of $(C,E,F)$. Shrinking $S$ if necessary, we may assume that over any geometric point $s\in S$, $\sE_s$ is periodic and $\sF_s\subset \sE_s$ is the maximal direct sum of degree zero Higgs sub line bundles. By maximality, $\sF_s$ must be periodic as well. Thus $F\subset E$ is a nontrivial periodic Higgs subbundle. If $F=E$, then (ii) follows from Corollary \ref{direct sum of line bundles}. Otherwise, $E/F$ is also a nontrivial periodic Higgs bundle by \cite[Proposition 6.2]{KS}. By induction hypothesis, both $F$ and $E/F$ are direct sums of torsion line bundles. We claim that the extension class of $E$, regarded as an extension of $E/F$ by $F$, is zero. Clearly, the claim completes the whole proof. Let $m$ be a common multiple of orders of torsion line bundles in $F$ and $E/F$. Then $m^*E$ becomes an extension of $(\sO_C^{\oplus r},0)$ by $(\sO_C^{\oplus s},0)$ for some $r,s\geq 1$, where $m: C\to C$ is the multiplication by $m$. It is \'etale, and therefore $m^*E$ is again periodic after \cite[Lemma 4.3]{KS}. Since the extension class of $m^*E$ remains nonzero, the periodicity of $m^*E$ contradicts the assertion in the last paragraph. The claim is proved.
	\end{proof}

	\section{Hyperbolic curves}
	In this section, we continue to explore the periodicity of $(E_{unif},\theta_{unif})$ in the case of hyperbolic curves. The findings of the first two subsections are analogous to those discovered in our early study of the projective line minus four points \cite{LS}, where we sought to link the periodicity of the the associated logarithmic uniformizing Higgs bundles to modularity of the underlying (affine) hyperbolic curves. In the last subsection, we resolve \cite[Conjecture 1.3]{SZZ}, that provides a mass formula for Shimura curves. The classical Deuring-Eichler mass formula may be viewed as a mass formula for modular curves. A finer study of the periodicity of $(E_{unif},\theta_{unif})$ in the case of Shimura curves leads to the notion of \emph{canonical Higgs-de Rham flows} over good reductions of Shimura curves. We view this canonical Higgs-de Rham flow over the good reduction of a Shimura curve as an algebro-geometric incarnation of 'uniformization'.

	\subsection{Non one-periodicity for generic curves}
	The following result is analogous to \cite[Proposition 1.5]{LS}.
	\begin{proposition}\label{non-periodicity}
		Let $C$ be a generic projective hyperbolic curve. Then $(E_{unif},\theta_{unif})$ over $C$ is non one-periodic.
	\end{proposition}
	The proof relies on several basic results established in \cite{LSYZ}, which we recall here for convenience of readers. Let $k$ be an algebraically closed field of positive characteristic $p$ and $C$ be a smooth projective curve of genus $g\geq 2$. In the affine space $H^1(C,K_C^{-p})$, there are two closed subvarieties $A$ and $K$: Let $\{[\tilde C]\}$ be the set of isomorphism classes of $W_2(k)$-liftings of $C$. It forms a torsor under $H^1(C,K_C^{-1})$. The bundle part of $C^{-1}_{C\subset \tilde C}(E_{unif},\theta_{unif})$, regarded as an extension of $K_C^{p/2}$ by $K_C^{-p/2}$, defines an element in $H^1(C,K_C^{-p})$. This element depends on $[\tilde C]$, and the collection of all images define an affine subspace $A\subset H^1(C,K_C^{-p})$ of dimension $3g-3$ (\cite[Lemma 3.2, 3.8]{LSYZ}). On the other hand, $K\subset H^1(C,K_C^{-p})$ is a certain cone of complementary dimension viz. $(p-1)(2g-2)$. By construction, the closed subset $A\cap K$ consists of those isomorphism classes of $W_2(k)$-liftings $\tilde C$ such that the maximal destabilzer of $C^{-1}_{C\subset \tilde C}(E_{unif},\theta_{unif})$ is of degree $g-1$ (\cite[Proposition 3.6]{LSYZ}). Moreover, the whole construction can be extended to a stable curve $[C]\in 
	\overline{\sM}_g(k)$ and even a relative stable curve over certain affine base (\cite[\S4.1]{LSYZ}). A crucial result is that, for a \emph{totally degenerate} stable curve, the set $A\cap K$ consists of one point (\cite[Proposition 4.7 ]{LSYZ}).  
	\begin{lemma}\label{A cap K}
		Let $\tilde C$ be a $W_2(k)$-lifting of $C$ such that $(K_C^{1/2}\oplus K_C^{-1/2},id)$ is one-periodic, where $id: K_C^{1/2}\to K^{-1/2}_{C}\otimes K_C$ is the tautological isomorphism. Then $[\tilde C]\in A\cap K$. 
	\end{lemma}
	\begin{proof}
		Set $(V,\nabla)=C^{-1}_{C\subset \tilde C}(K_C^{1/2}\oplus K_C^{-1/2},id)$. One-periodicity of $(K_C^{1/2}\oplus K_C^{-1/2},id)$ implies that we have the following short exact sequence:
		$$
		0\to K_C^{1/2}\to V\to K_C^{-1/2}\to 0.
		$$
		
		For any subline bundle $L\subset V$, we obtain an $\sO_C$-linear morphism $\theta$ which is the composite of natural morphisms
		$$
		\theta:  L\to V\stackrel{\nabla}{\longrightarrow}V\otimes K_C\to V/L\otimes K_C. 
		$$
		If $\theta\neq 0$, then it follows that $\deg(L)\leq \deg(V/L\otimes K_C)$ and hence $\deg L\leq g-1$. If $\theta=0$, then $L$ is $\nabla$-invariant. Then there is a corresponding $\theta$-invariant subline bundle $M\subset K_C^{1/2}\oplus K_C^{-1/2}$. Clearly $\theta|_{M}=0$. It follows that $L\cong M^{\otimes p}$. By stability, $\deg(M)<0$, which implies that $\deg(L)<0$. Therefore, the maximal destabilizer of $V$ is of degree $g-1$, which is precisely the statement that $[\tilde C]\in A\cap K$.
	\end{proof}
	\begin{proof}[Proof of Proposition \ref{non-periodicity}]
		Let $(\sC, \sE_{unif},\Theta_{unif})/S$ be a spreading-out of $(C,E_{unif},\theta_{unif})$. As $C$ is generic, the associated moduli map $S\to \sM_{g}$ to the moduli stack of genus $g$ over $\Z$ is dominant. Shrinking $S$ if necessary, we may assume that moduli map factors as 
		$$
		S\stackrel{\phi}{\longrightarrow} U\hookrightarrow \sM_g,
		$$
		where $\phi$ is smooth surjective, and $U$ is a nonempty open subset of $\sM_g$. Let $p$ be a prime number such that $U_p:=U\times \bar \F_p\neq \emptyset$. We claim that there exists a nonempty open subset $V_p\subset U_p$ such that for any $[C]\in V_p$, the $A\cap K$ discussed above consists of finitely many points. The claim concludes the proof: Take any $s\in S$ such that $[C]=\phi(s)\in V_p$ (that is, $\sC_s\cong C$). By Lemma \ref{A cap K} and the paragraph before it, there exists a $W_2(\bar \F_p)$-lifting $\tilde C$ such that $(K_C^{1/2}\oplus K_C^{-1/2},id)$ is not one-periodic with respect to it. As $\phi$ is smooth, there exsits some $W_2(\bar \F_p)$-lifting $\tilde s\to S$ of $s$ such that $\sC_{\tilde s}\cong \tilde C$. It follows that $(\sE_{unif},\Theta_{unif})_s$ is not one-periodic with respect to $\tilde s$. This proves that $(E_{unif},\theta_{unif})$ is not one-periodic, as desired.
		
		To show the claim, we modify the degeneration argument in \cite[Proposition 4.9]{LSYZ}. First, we take a one-dimensional local deformation $f: \sC_1\to B_1$ with $f^{-1}(0)$ a totally degenerate curve and $f^{-1}(b_1)$ smooth for any $b_1\neq 0$. Associated to $f$, we have the relative $\sA$ and $\sK$ inside $\sV=R^1f_*(F^*\sT_{\sC_1/B_1})$ (where $F$ is the relative Frobenius). Consider the composite
		$$
		\pi_f: \sA\cap \sK\subset \sV\to B_1.
		$$
		Since the fiber $\pi_f^{-1}(0)$ consists of only one point, it follows by semicontinuity that there exists $b_1\neq 0$ such that $\pi^{-1}(b_1)$ consists of finitely many points. Second, we take a local deformation $g: \sC_2\to B_2$ with $b_1\in B_2$ and $B_2$ a nonempty affine open subset of $\sM_g\times \bar \F_p$. Now we apply the semicontinuity to the morphism
		$$
		\pi_g: \sA\cap \sK\to B_2
		$$
		with $\pi_g^{-1}(b_1)$ consists of finitely many points, and obtain that there exists a nonempty open subset of $B_2$ over which closed fibers of $\pi_g$ are finitely many points. The claim follows.
	\end{proof}

	\subsection{Motivicity and modular embeddings}\label{section:CS}
	Let $C$ be a smooth projective hyperbolic curve. The uniformization theorem asserts 
	$$
	C^{an}\cong \Gamma\backslash \HH,
	$$
	where $\Gamma\subset \PSL_2(\R)$ is a torsion free lattice. Such a uniformization induces an injective group homomorphism $\alpha: \pi_1(C^{an})\to \PSL_2(\R)$, where $\PSL_2(\R)$ is considered as the group of biholomorphisms of $\HH$. A \emph{theta characteristic} is a lift $\tilde \alpha: \pi_1(C^{an})\to \SL_2(\R)$ so that its composition with $\SL_2(\R)\to \PSL_2(\R)$ is $\alpha$. Theta characteristics exist, and there are $2^{2g}$ theta characteristics in total. Let $\tilde \Gamma$ be the image of $\tilde \alpha$. $\Gamma$ (or $\tilde \Gamma$) is an \emph{arithmetic Fuchsian group} if $\tilde \Gamma$ is commensurable with $\sO_D^{N=1}$, where $D$ is a quaternion (division) algebra over a totally real field $F$ which is split at only one real place of $F$, $\sO_D$ is a maximal order, and the superscript $N=1$ denotes reduced norm 1. The resulting curve is called a \emph{connected Shimura curve}, which we shall study in more detail in next subsection. Motivated by a question in transcendental number theory, Cohen-Wolfart \cite{Wolf} introduced the notion of a \emph{modular embedding} for a cofinite (not necessarily cocompact) Fuchsian group. Modular embeddings generalize the notion of an arithmetic Fuchsian group, hence the notion of a modular/Shimura curve in a natural way. Let us recall its definition. 
	
	As above, let $F$ be a totally real field of degree $d$ over $\Q$ and let $D$ be a quaternion algebra over $F$. We assume that $D$ is split at at least one real place of $F$, and we index the embeddings $\tau_i\colon F\hookrightarrow \R$ so that only the places $\tau_1,\dots,\tau_r$ split $D$. For every $1\leq j \leq r$ we fix an isomorphism 
	$$
	s_j\colon D\otimes_{\tau_j}\R\cong M_{2\times 2}(\R).
	$$
	Fix an order $\sO\subset D$ and let $\sO^1$ be the invertible elements of reduced norm 1. By the isomorphisms we just chose $\sO^1$ embeds as an irreducible arithmetic lattice in $\SL_2(\R)^r$. Let  $\triangle\subset \PSL_2(\R)^r$ be a subgroup commensurable to the image of $\sO^1$ in $\PSL_2(\R)^r$. The following definition is taken from Definition 1.1 \cite{Kuc}.
	\begin{definition}
		Let $\triangle$ be as above. Let $\Gamma\subset \PSL_2(\R)$ be a cofinite Fuchsian subgroup and let $1\leq i\leq r$. A modular embedding for $\Gamma$ with respect to $\tau_i$ is a pair $(\phi, \tilde f)$ such that 
		\begin{itemize}
			\item [(i)] $\phi: \Gamma\to \triangle$ is a group homomorphism such that the composite
			$$
			\Gamma\stackrel{\phi}{\to} \triangle\subset \textrm{PSl}_2(\R)^r \stackrel{pr_i}{\to}  \textrm{PSl}_2(\R)
			$$
			is the identity. 
			\item [(ii)]  $\tilde f: \HH \to \HH^r$ is an equivariant holomorphic map with respect to $\phi$, i.e. satisfying
			$$
			\tilde f(\gamma z)=\phi(\gamma)\tilde f(z), \quad \forall z\in \HH, \gamma\in \Gamma.
			$$
		\end{itemize}
	\end{definition}
	A $\Gamma$ admitting a modular embedding withe respect to $\tau_i$ does not necessarily admit a modular embedding with respect to $\tau_j$ for $j\neq i$. However, the set of \emph{all} $\Gamma$s admitting modular embeddings with respect $\tau_i$ equals the corresponding set with respect to $\tau_j$ (In this sense, \cite[Definition 1.1]{Kuc} is equivalent to \cite[Definition 4]{SM}). Looking into the definition, one may ask the following question: does a subgroup of $\textrm{PSL}_2(\R)$, commensurable to $\Gamma$ which admits a modular embedding, admit also a modular embedding? It is tautologically true for $r=1$ (i.e., the case of compact Shimura curves), but \`a priori unclear for $r\geq 2$. 
	
	Returning to uniformization theory, we say a smooth projective hyperbolic curve $C$ \emph{admits a modular embedding} if the image of $\pi_1(C^{an})$ under $\alpha$ admits a modular embedding with respect to \emph{some} $\tau_i$. This notion is independent of the choice of a uniformization isomorphism for $C^{an}$. We make the following
	\begin{conjecture}\label{conjecture on modular embedding}
		Let $C$ be a smooth projective hyperbolic curve. Then $C$ admits a modular embedding if and only if its uniformizing Higgs bundle is periodic. 
	\end{conjecture}
	The \emph{periodic Higgs conjecture} of the last section of this article would imply the following: a curve $C$ have a periodic uniformizing Higgs bundle implies that it admits a modular embedding.There is an analogous conjecture for affine hyperbolic curve, but we shall not discuss it here (see the earlier work \cite{LS} on $\P^1$ minus four points). In this subsection, we shall p rove the following result.
	\begin{theorem}\label{intro_geo_modular}
		Let $C$ be a smooth projective hyperbolic curve. Then the following are equivalent.
		\begin{enumerate}
			\item $C$ admits a modular embedding.
			\item $(E_{unif},\theta_{unif})$ over $C$ is motivic.
		\end{enumerate}
	\end{theorem}
	By the Higgs periodicity theorem, the above theorem implies the only-if part of Conjecture \ref{conjecture on modular embedding}.  
	
	\begin{corollary}\label{commensurability}
		Let $\Gamma\subset \PSL_2(\R)$ be a torsion free cocompact lattice which admits a modular embedding. Then any torsion free cocompact lattice of $\PSL_2(\R)$ which is commensurable to $\Gamma$ also admits a modular embedding. 
	\end{corollary}
	\begin{proof}
		Let $\Gamma'$ be a torsion free cocompact lattice, commensurable to $\Gamma$. Set $\Gamma''$ to be their intersection. Associated to $\Gamma, \Gamma',\Gamma''$, there are three smooth projective hyperbolic curves $C,C',C''$ whose analytifications are isomorphic to $\Gamma\backslash \HH, \Gamma'\backslash \HH, \Gamma''\backslash \HH$ respectively. By construction, there is a correspondence  
		\[\xymatrix{ & C''\ar[dl]_{\pi_C} \ar[dr]^{\pi_{C'}} & \\ C & & C'}\]
		with finite \'etale morphisms $\pi_{C}$ and $\pi_{C'}$. Since $C$ admits a modular embedding, it follows from Theorem \ref{intro_geo_modular} that, there exists a smooth projective morphism $g: X\to C$ such that a uniformizing Higgs bundle of $C$ appears as a direct summand of a Kodaira-Spencer system $(E,\theta)_g$ attached to $g$.\footnote{Indeed, one may choose the family $X\to C$ to be an abelian scheme.} Pulling back $g$ along $\pi_{C}$, we obtain a smooth projective morphism $g'': X''=X\times_{C}C''\to C''$. Since a uniformizing Higgs bundle pulls back to a uniformizing Higgs bundle along a finite \'etale morphism, it is easy to see that a uniformizing Higgs bundle of $C''$ appears in the Kodaira-Spencer system $(E,\theta)_{g''}\cong \pi_{C}^*(E,\theta)_g$ attached to $g''$. By taking a further \'etale morphism of degree two $\widetilde{C''}\to C''$ if necessary, we may assume that that uniformizing Higgs bundle of $C''$ in $(E,\theta)_{g''}$ is the pull-back of a uniformizing Higgs bundle of $C'$ along $\pi_{C'}$. Set $g'=\pi_{C'}\circ g'': X''\to C'$. It is smooth and projective as $\pi_{C'}$ is so. By the projection formula, we see that that uniformizing Higgs bundle of $C'$ appears in the Kodaira-Spencer system $(E,\theta)_{g'}$ attached to $g'$. Since both a uniformizing Higgs bundle and a Kodaira-Spencer system have degree zero, it is a direct summand by the polystability of a Kodaira-Spencer system. The result now follows from Theorem \ref{intro_geo_modular}.
	\end{proof}
	
	To proceed with the proof of Theorem \ref{intro_geo_modular}, we need to interpret modular embeddings in geometric terms. This is exactly what Kucharczyk initiated in his Ph.D thesis \cite{Kuc}. His first main result \cite[Theorem A.1]{Kuc} (see also Proposition 2.8 loc. cit.) provides a key ingredient in the proof of one direction in Theorem \ref{intro_geo_modular}. 
	
	\begin{lemma}\label{modular_characterization} 
		Notation as above. If $C$ admits a modular embedding, then the uniformizing Higgs bundle $(E_{unif},\theta_{unif})$ of $C$ is motivic.		
	\end{lemma}
	\begin{proof}
		By definition, there is a totally real field $F$, a quaternion algebra $D$ over $F$, an order $\sO\subset D$, a subgroup $\triangle\subset \PSL_2(\R)^r$, commensurable to the image of $\sO^1$, such that for some $1\leq i_0\leq r$, the uniformizing representation of $\alpha:\pi_1(C^{an})\to \PSL_2(\R)$ factors as
		$$
		\pi_1(C^{an})\to \triangle\to \PSL_2(\R)^r\stackrel{pr_{i_0}}{\to} \PSL_2(\R).
		$$
		Let $\overline{\sO^1}\subset \textrm{PSl}_2(\R)^r$ be the image of $\sO^1$. Let $\tilde \triangle\subset \sO^1$
		be the inverse image of $\triangle\cap \overline{\sO^1}$. For any torsion free subgroup $ \triangle'\subset \tilde \triangle$ of finite index, the quotient analytic space $S':= \triangle'\backslash\HH^r$ is a smooth complex algebraic variety. Recall that $s_j\colon D\otimes_{\tau_j}\R\cong M_{2\times 2}(\R)$. Via the chosen isomorphisms $s_j$s, there are tautological rank two $\R$-local systems $\L_i,1\leq i\leq r$ over $S'$, defined by the following composite map
		$$
		\pi_1(S'^{an})\cong  \triangle'\subset \sO^1\subset \SL_2(\R)^r\stackrel{pr_i}{\to} \SL_2(\R).
		$$
		Each complex local system $\L_i\otimes_{\R}\C$ is an irreducible $\C$-VHS (\cite[\S4]{Sim92}). If $D$ is split viz. $D\cong M_2(F)$, $S'$ is a Hilbert modular variety. It is not difficult to construct a universal family of abelian varieties $g: \sA\to S'$ of dimension $d$ with real multiplication by $F$, whose associated complex monodromy representation is just $\bigoplus_{1\leq i\leq r}\L_i\otimes \C$. If $D$ is a division algebra, one uses the so-called Deligne's mod\`ele \'etrange (see Appendix A). This yields a (non-canonical) family of abelian varieties $g: \sA\to S'$ of dimension 4d with complex multiplication by $E$, a totally imaginary quadratic extension of $F$, whose associated complex monodromy representation contains $\bigoplus_{1\leq i\leq r}\L_i\otimes \C$ as a direct summand. Now that $\triangle'$ is torsion free, it maps injectively onto a subgroup of $\triangle$, which we again denote by by $\triangle'$. As $\alpha$ is injective, we may regard $\pi_1(C^{an})$ as a subgroup of $\triangle$. Note that the subgroup $\pi_1(C^{an})\cap \triangle'\subset \pi_1(C^{an})$ is of finite index. Therefore, by Riemann's existence theorem, this finite index subgroup gives rise to a finite \'etale morphism $\pi: C'\to C$, such that the induced map $\pi_{*}: \pi_1(C'^{an})\to \pi_1(C^{an})$ is the previous inclusion. By construction, we obtain a theta characteristic for $C'$, that is, the composite
		$$
		\pi_1(C'^{an})\to \triangle'\to \SL_2(\R)\to \PSL_2(\R)
		$$    
		is a uniformizing representation of $C'$. Moreover, the modular embedding for $C$ gives rise to a modular embedding for $C'$, which yields a morphism $f': C'\to S'$. By \cite[Theorem A1]{Kuc} (and remarks following the statement in loc. cit.), a uniformization Higgs bundle (equivalently, a graded rank two Higgs bundle with a maximal Higgs field) of $C'$ appears in the Kodaira-Spencer system attached to the pull-back family $\sA\times_{S'}C'\to C'$ via $f'$. Therefore, $(E_{unif},\theta_{unif})$ of $C'$ is motivic. By the argument in the proof of Corollary \ref{commensurability}, $(E_{unif},\theta_{unif})$ of $C'$ is motivic too.  
	\end{proof}

	The next result is derived from the study of rank two local systems over quasi-projective manifolds due to Corlette-Simpson \cite{CS}, especially \S9-\S10 loc. cit. See also \cite[Theorem 5]{Sim92} and its proof.  Our input is to explicate the fact that the polydisk Shimura modular varieties, as introduced by Corlette-Simpson, are connected quaternionic Shimura varieties.  For further details on connected quaternion Shimura varieties, see Appendix A.
	
	\begin{proposition}[Corlette-Simpson]\label{real_simpson}
		Let $\V_{\sO_A}$ be a local system over $C^{an}$ of projective $\sO_A$-module of rank two for an algebraic number field $A\subset \C$. Suppose that $\V:=\V_{\sO_A}\otimes_{A} \C$ is non-unitary and for any embedding $\sigma: A\to \C$, the $\sigma$-conjugate complex local system $\V^{\sigma}=\V_{\sO_A}\otimes_{A,\sigma}\C$ underlies a $\C$-VHS. Then there exists a non-constant morphism $f: C \to S$, where $S=\triangle\backslash \HH^r, r>0$ is a connected quaternionic Shimura variety, such that $\V\cong f^*\L_i\otimes_{\R} \C$ for some $1\leq i\leq r$. 
	\end{proposition}
	\begin{proof}
		By  \cite[Lemma 8.3]{CS}, $\V$ is a direct factor of the complexified weight one $\Z$-VHS $\W=\V^{\oplus 2}$ associated to a family of abelian varieties over $C$. It follows from a lemma of M. Larsen \cite[Lemma 4.8]{Sim92} that $\V$ is defined over a totally imaginary quadratic extension $L\subset \C$ of a totally real field $F$. That is, there is a rank two $L$-local system $\V_L$ such that $\V_L\otimes \C= \V$.  Since $\V$ comes from $\V_{\sO_A}$, $\V$ is integral. By Bass-Serre theory, there exists a projective $\sO_L$-sub local system $\V_{\sO_L}\subset \V_L$ such that $\V_{\sO_L}\otimes L=\V_L$.  Clearly, for each embedding $\sigma: L\to \C$, $\V^{\sigma}_L=\V_L\otimes _{L,\sigma}\C$ is a direct factor of $\W\otimes \C$, and hence underlies a $\C$-VHS. In other words, we may replace the number field $A$ in the proposition by a totally imaginary field $L$. Fix a base point $x\in C$.  Now by \cite[Proposition 10.2, Lemma 10.3, Lemma 10.4]{CS}, there exists a structure of $\C$-VHS of weight one on $\V^{\sigma}$ for each $\sigma$, and a $\pi_1(C^{an})$-invariant antisymmetric sesquilinear form $\Phi: P\times P \to \sO_L$, where $P=(\V_{\sO_L})_x$, such that $\V_{\sO_L}$ becomes a VHS of type $(P,\Phi)$.  Thus, by \cite[Theorem 9.2]{CS}, there exists a map $f: C \to S$, where $S$ is one of polydisk Shimura modular varieties, such that $\V_{\sO_L}\cong f^*\sV$, where $\sV$ is the tautological VHS of type $(P,\Phi)$ over $S$. If either $r=0$ or $f$ is constant, $\V_{\sO_L}$ would be trivial, contradicting the non-unitarity of $\V$. So $r>0$ and $f$ is non-constant.  Finally, by Lemma \ref{polydisk} and its proof, there exists some tautological rank two $\C$-local system $\M_i$ over $S$ such that $\V\cong f^*\M_i$.  Since $\V$ is non-unitary, $\M_i$ must be non-unitary. Hence $\M_i\cong \L_i\otimes _{\R}\C$ for some tautological rank two $\R$-local system in the proof of Lemma \ref{modular_characterization}.
	\end{proof}
	
	Now we may proceed to the proof of Theorem \ref{intro_geo_modular}.
	\begin{proof}
		After Lemma \ref{modular_characterization}, it remains to show (2) implies (1).  Let $\V$ be a uniformizing local system over $C^{an}$ whose complexification $\V_{\C}$ corresponds to a uniformizing Higgs bundle $(E_{unif},\theta_{unif})$ associated to $C$. Then $\V_{\C}$ is irreducible as $(E_{unif},\theta_{unif})$ is stable. By the motivicity of $(E_{unif},\theta_{unif})$ , the restriction $\V|_{U}$ to some nonempty Zariski open subset is integral, and each of its Galois conjugations underlies a $\C$-VHS. Then $\V$ has the same properties. (Note that a $\C$-VHS structure over a Galois conjugation of $\V|_{U}$ extends uniquely to a $\C$-VHS structure on the corresponding Galois conjugation of $\V$.) By Proposition \ref{real_simpson}, there is a morphism $f: C\to S$, where $S$ is a connected quaternionic Shimura variety, and some $i$, such that $\V_{\C}\cong f^*\L_i\otimes \C$. Therefore, $\V\cong f^*\L_i$ as real local system. It follows that $C$ admits a modular embedding.  
	\end{proof}

	\subsection{Mass formula of Shimura curves}
	In this subsection, we propose a new approach to study Newton stratifications of Shimura varieties of Hodge type. Although our result is restricted only to the curve case, the non-abelian Hodge theoretical part of our method can be generalized to higher dimensinal Shimura varieties in a quite straightforward manner.  In particular, we expect the existence of the canonical periodic Higgs-de Rham flows over Shimura varieties as motivated by the theory of canonical $\C$-VHS of Calabi-Yau type over locally symmetric domains (see \cite{SZ12}).
	
	In this subsection, we work on a detailed description of the canonical periodic Higgs-de Rham flows over Shimura curves. From this canonical flow, we deduce a Hasse-Witt map attached to a Shimura curve; from this latter map we are able to derive a generalization of Deuring's mass formula.

	Let $\sM_0/k$ be a geometrically connected component of the good reduction of a Shimura curve as in Notation \ref{shimura_curve}, with $F$ a totally real field and $\mathfrak p$ a prime of $F$. 
	We will compute the period of $(E_{unif},\theta_{unif})$ over $\sM_0$, and use this periodicity to construct the Hasse-Witt map attached to $\sM_0$. Let $\tilde {\sM_0}$ be the $W_2(k)$-scheme coming from the reduction of the global Shimura curve modulo $\mathfrak p^2$.
	
	\begin{theorem}\label{theorem on periodicity of the rank two Higgs bundle}
		
		Notation as above. Let $f=[F_{\mathfrak p}:\Q_p]$. Then, with respect to the $W_2$-lifting $\tilde {\sM_0}$ of ${\sM_0}$, the Higgs bundle $(E_{unif},\theta_{unif})$ is $f$-periodic. More precisely, we have the following diagram of a periodic Higgs-de Rham flow:
		
		$$
		\xymatrix{
			& (H,\nabla)\ar[dr]^{Gr_{Fil_{HN}}} && (F_{{\sM_0}}^{*f-1}H,\nabla_{can})\ar[dr]^{Gr_{Fil_{HN}}} \\
			(E_0,\theta_0)=(E_{unif},\theta_{unif})\ar[ur]^{C^{-1}} & & \cdots\cdot \ar[ur]^{C^{-1}}&& (E_f,\theta_f),\ar@/^2pc/[llll]^{\stackrel{\psi}{\cong} } }
		$$
		where for $0\leq i\leq f-2$, the bundles $F_{\sM_0}^{*i}H$ are stable (hence the Harder-Narasimhan filtration $Fil_{HN}$ over them are trivial), and $F_{\sM_0}^{*f-1}H$ is unstable (hence the Harder-Narasimhan filtration over it is nontrivial).
		
	\end{theorem}
	
	\begin{remark}\label{remark_on_theta}
		
		All of the terms in the periodic Higgs de-Rham flow are essentially determined by the curve $\tilde {\sM_0}$ over $W_2$. Note that $f$ is the first natural number such that the $f\textsuperscript{th}$-iterated Frobenius pullback of $H$ become unstable. Also, since the Higgs bundle $(E_{unif},\theta_{unif})$ is Higgs stable, the isomorphism $\psi$ of graded Higgs bundles is unique up to scalar. Furthermore, for another choice of $(E_{unif},\theta_{unif})$, the diagram will differ by tensoring with a two-torsion line bundle.
		
	\end{remark}
	
	\begin{proof} The periodicity of $(E_{unif},\theta_{unif})$ is an intrinsic property of the Shimura curve (over $W_2$). However, in order to show this property, we need to make use of Deligne's mod\`ele \'etrange. Choose an auxiliary imaginary quadratic extension $\Q(\alpha)$ where $p$ splits. One gets a new Shimura curve of PEL type from $B:=D\otimes_F F(\alpha)$. In particular, there is an abelian scheme $\scrU_0\rightarrow \scrM_0$ over the integral model, which we may reduce modulo $p$ to obtain $\sU_0\rightarrow \sM_0$. We have the following property by \cite[2.6.2 (b)]{Car}: there exists an order $\sO\subset B$, maximal at $\mathfrak p$, such that
		
		$$\sO\hookrightarrow\End_{\sM_0}(\sU_0).$$

		We will use several results from \cite{SZZ}, specifically Corollary 4.2, Proposition 4.3, 4.4 and Theorem 7.3. Notice that neither the assumption $p\geq 2g({\sM_0})$ in Proposition 4.4 nor the assumption $p\geq \max\{2g({\sM_0}),2([F:\Q]+1)\}$ in Theorem 7.3 of \emph{loc. cit} are necessary. This is because of the explicit reconstruction of the inverse Cartier transform of Ogus-Vologodsky in terms of exponential twisting (see \cite{LSZ0}). Let $p\sO_F=\prod_{i}\mathfrak p_i$ be the prime decomposition, with $\mathfrak p=\mathfrak p_1$. Choose an isomorphism $\C\to \bar \Q_p$ such that under the induced identification	
		$$
		\Hom(F,\bar \C)=\coprod_{i}\Hom_{\Q_p}(F_{\mathfrak p_i}, \bar \Q_p),
		$$
		$\tau$ is sent to $\Hom_{\Q_p}(F_{\mathfrak p}, \bar \Q_p)$. The graded Higgs bundle $(E,\theta)$ attached to the family $f_0: \sU_0\to \sM_0$ over $\bar \F_p$ is one-periodic by \cite[Proposition 4.1]{LSZ0}. That is, there is a natural isomorphism of graded Higgs bundles	
		$$
		Gr_{F_{hod}}\circ C^{-1}(E,\theta)\cong (E,\theta),
		$$	
		where $F_{hod}$ denotes for the Hodge filtration on the degree one de Rham bundle attached to $f_0$. The Higgs bundle $(E,\theta)$ is irreducible as an $\sO$-module. Pick a subfield $L\subset D$, containing $F$, that splits $D$ and leaves $\mathfrak p$ inert. Set $\sO':=\sO\cap L\otimes_F F(\alpha)\subset \End_{\sM_0}(\sA_0)$. With respect to $\sO'$ the Higgs bundle decomposes into a direct sum of rank two graded Higgs subbundles:
		$$
		(E,\theta)=\bigoplus_{\phi\in \Hom(L,\bar \C)}(E_{\phi},\theta_{\phi})\oplus (E_{\bar \phi},\theta_{\bar \phi})
		$$	
		by \cite[Proposition 4.1]{SZZ}. We explain the notation. Fix the diagram $L\supset F\subset F(\alpha)$ and an embedding $\Q(\alpha)\hookrightarrow \C$. Then we have a map $\Hom(L,\C)\hookrightarrow \Hom(L\otimes_FF(\alpha),\C)$ given by sending $\alpha$ to the chosen image in $\C$. Given $\phi\in \Hom(L,\C)$, identify it with the induced map $\Hom(L\otimes_F F(\alpha),\C)$; then $\bar \phi$ is the conjugate of $\phi$ with respect to $\alpha$.	
		
		Now, let $\sigma$ denote the absolute Frobenius of the base field $\overline{\F}_p$ and let $*$ denote the involution of $L/F$. Since the inverse Cartier transform $C^{-1}$ is $\sigma$-linear and $Gr_{F_{hod}}$ is linear, the flow operator $Gr_{F_{hod}}\circ C^{-1}$ transforms the $\phi$-factor to the $\sigma(\phi)$-factor. The $\sigma$-orbit of a $\phi_1$ with the property $\phi_{1}|_{F}=\tau$ reads
		
		$$
		\Hom_{\Q_p}(L_{\mathfrak p},\bar \Q_p)=\{\phi_1,\cdots,\phi_f,\phi_1^*,\cdots,\phi_f^*\}.
		$$
		
		It follows from \cite[Proposition 4.3]{SZZ} that
		
		$$
		(Gr_{F_{hod}}\circ C^{-1})^{f}(E_{\phi_1},\theta_{\phi_1})=(E_{\phi_1^*},\theta_{\phi_1^*})\cong (E_{\phi_1},\theta_{\phi_1}).
		$$
		
		By the argument of \cite[Proposition 4.4]{SZZ}, it follows that $(E_{\phi_1},\theta_{\phi_1})$ is of maximal Higgs field and therefore differs from our chosen uniformizing Higgs bundle by a two-torsion line bundle. Therefore $(E_{unif},\theta_{unif})$ initiates an $f$-periodic flow.

		Finally, we describe the terms of this flow intrinsically in terms of $\sM_0$. By the analysis of the Higgs subbundles in \cite[Proposition 4.1]{SZZ}, one knows that the Higgs fields of	
		$$
		(Gr_{F_{hod}}\circ C^{-1})^{i}(E_{\phi_1},\theta_{\phi_1}), 1\leq i\leq f-1,
		$$
		are zero. By \cite[Proposition 6.6 (ii)]{SZZ}, $F_{\sM_0}^{*i}H, 0\leq i\leq f-2$ is stable, and $F_{\sM_0}^{*f-1}H$ becomes unstable. By \cite[Corollary 7.4 ]{SZZ}, the Hodge filtration on $(F_{\sM_0}^{*f-1}H,\nabla_{can})$ is nothing but the Harder-Narasimhan filtration.
		
	\end{proof}

	{\itshape Construction of the Hasse-Witt map:} Set $\sL=K_{\sM_0}^{1/2}$. Applying the functor $C^{-1}$ to the Higgs subbundle $(\sL^{-1},0)\subset (E_{unif},\theta_{unif})$, we obtain $(F_{\sM_0}^*\sL^{-1},\nabla_{can})\subset (H,\nabla)$. In particular, $F_{\sM_0}^*\sL^{-1}\subset H$ and consequently $F_{\sM_0}^{*f}\sL^{-1}\subset F_{\sM_0}^{*f-1}H$. Since $$Gr_{Fil_{HN}}(F_{\sM_0}^{*f-1}H,\nabla_{can})\cong (E_{unif},\theta_{unif}),$$ it follows that
	$$
	0\to \sL\to F_{\sM_0}^{*f-1}H\to \sL^{-1}\to 0
	$$	
	is a short exact sequence. The Hasse-Witt map is defined to be the composite	
	$$
	HW(\sM_0): F_{\sM_0}^{*f}\sL^{-1}\hookrightarrow F_{\sM_0}^{*f-1}H\twoheadrightarrow \sL^{-1}.
	$$
	
	We may also regard $HW(\sM_0)$ as a global section of the line bundle $F_{\sM_0}^{*f}\sL\otimes \sL^{-1}\cong \sL^{p^f-1}$. Since $p$ is odd, the bundle $\sL^{p^f-1}$ is independent of the choice $K_{\sM_0}^{1/2}$ and by Remark \ref{remark_on_theta}, the section $HW({\sM_0})\in \Gamma(\sM_0,\sL^{p^f-1})$ is also independent of this choice. Note also that $p$ being odd means that $HW(\sM_0)$ is a pluricanonical differential form; in other words, it is a modular form modulo $p$. We shall call this natural section the \emph{Hasse-Witt invariant} of ${\sM_0}$.

	At this point, by considering the zero locus of $HW$, we have obtained a canonical divisor of $\mathcal M_0$. Our next goal is to show that the Newton polygon of $\mathcal U_0\rightarrow \mathcal M_0$ is constant over this divisor, and to moreover compute the degree of this divisor.
	
	The second ingredient in our approach to the mass formula is the so-called ``one clump theorem" \cite{Kr}, whose origin is traced back to the work \cite{Mo} of S. Mochizuki.  Its usage here stems from the observation that the Newton jumping locus is invariant under a Hecke correspondence. This phenomenon has been abstracted into the notion of \emph{clump} of an \'etale correspondence without a core. The ``one clump theorem" then implies the Newton jumping locus (under any symplectic embedding) is precisely the zero-locus of the Hasse-Witt map established as above. We first collect (and modify) several results in \cite{Kr} that we will use for the proof of the mass formula.
	
	\begin{definition} Let $X,Y,Z$ be (not necessarily geometrically connected) curves over $k$. A \emph{correspondence of curves over $k$} is a diagram \[\xymatrix{ & Z\ar[dl]_f \ar[dr]^g & \\ X & & Y}\]where $f$ and $g$ are finite, dominant and generically separable.
		
	\end{definition}


	\begin{definition} \label{nocore}A correspondence of geometrically connected curves $X\leftarrow Z\rightarrow Y$ over $k$ \emph{has no core} if $$k\cong k(X)\cap_{k(Z)} k(Y).$$
		
	\end{definition}
	
	\begin{remark}By \cite[Prop. 4.2]{Kr}, the property of ``having a core" is unchanged under arbitrary field extension $L/k$.
		
	\end{remark}

	\begin{definition} \label{clump}Let $X\leftarrow Z\rightarrow Y$ be a correspondence of curves over $k$. A \emph{clump} is a finite set $S\subset Z(\bar{k})$ such that $f^{-1}(f(S))=g^{-1}(g(S))=S$. A clump $S$ is \'etale if $f$ and $g$ are \'etale at all points $s\in S$.
		
	\end{definition}
	
	We cite \cite[Theorem 9.6]{Kr} as follows.
	\begin{theorem} \label{oneclump}  
		Let $X\leftarrow Z\rightarrow Y$ be a correspondence of geometrically connected curves over $k$ without a core. Then there is at most one \'etale clump.
	\end{theorem}
	
	Theorem \ref{oneclump} is a main input in proving Theorem \ref{main_result}. We will briefly introduce the main concepts that are relevant to the ``one clump theorem".
	
	\begin{definition}Let \[\xymatrix{ & Z\ar[dl]_f \ar[dr]^g & \\ X & & Y}\] be a correspondence of geometrically connected curves over $k$. An \emph{invariant line bundle} $\scrL$ is a triple $(\sL_X,\sL_Y,\phi)$ where $\sL_X$ is a line bundle on $X$, $\sL_Y$ is a line bundle on $Y$, and $\phi:f^*\sL_X\rightarrow g^*\sL_Y$ is an isomorphism of line bundles on $Z$. An \emph{invariant section} of $\scrL$ is a pair $(s_X,s_Y)\in H^0(X,\sL_X)\oplus H^0(Y,\sL_Y)$ such that $f^*s_X=\phi^*g^*s_Y$. The space of invariant sections is denoted by $H^0(\scrL)$ and its dimension is denoted by $h^0(\scrL)$.
		
	\end{definition}
	
	Then the one clump theorem essentially follows from the following theorem, which is \cite[Proposition 8.2, Corollary 8.10]{Kr}:
	
	\begin{theorem}Let $X\leftarrow Z\rightarrow Y$ be a correspondence of curves over $k$ without a core. Let $\scrL$ be an invariant line bundle. Then $h^0(\scrL)\leq 1$. If $\scrL$ and $\scrM$ are invariant line bundles of positive degree, then there exists positive integers $m,n$ such that $$\scrL^{\otimes m}\cong \scrM^{\otimes n}$$
		
	\end{theorem}
	
	The utility of this concept is the following: the Hasse-Witt map $HW(\sM_0)$ yields an \emph{invariant pluricanonical differential form} on the good reduction of a Hecke corresondence of Shimura curves. We will need several auxiliary results about \'etale correspondences without a core.
	
	\begin{lemma}\label{finitegaloiscore}Let $X\leftarrow Z\rightarrow Y$ be an \'etale correspondence of geometrically connected curves over $k$, where $Z$ is \emph{hyperbolic}. Then the following are equivalent:
		
		\begin{enumerate}
			
			\item $X\leftarrow Z\rightarrow Y$ has a core.
			
			\item There exists a finite extension $l/k$, a geometrically connected curve $W/l$, and a finite \'etale map $W\rightarrow Z_l$ such that the induced maps $W\rightarrow X_l$ and $W\rightarrow Y_l$ are Galois.
			
		\end{enumerate}
		
	\end{lemma}
	
	\begin{proof}
		
		\cite[Lemma 4.2]{Kr} implies that $X\leftarrow Z\rightarrow Y$ has a core if and only if there exists a finite field extension $l/k$ and a finite, generically separable map $W'\rightarrow Z_l$ such that the induced maps to $X_l$ and to $Y_l$ are Galois. Here, ``Galois" means on the level of function fields. In particular, if such $W$ exists as in (1), then the correspondence has a core. On the other hand, if there is a core, then there is a finite extension $l/k$ and a curve $W'\rightarrow Z_l$ that is finite Galois over $X_l$ and $Y_l$ (in the sense of function fields), again by Lemma 4.2 of \emph{loc. cit.} Given any such morphism, there is a unique factorization $W'\rightarrow W\rightarrow Z_l$ where the first map is a ramified covering of curves and the second arrow a maximal finite \'etale subcover. The uniqueness implies that $W\rightarrow X_l$ and $W\rightarrow Y_l$ are Galois, as desired.
		
	\end{proof}
	
	\begin{lemma}\label{nocorespecializes}Let $S=\text{Spec}(R)$ be the spectrum of a discrete valuation ring with generic point $\eta$ and special point $s$. Let $\sX$, $\sY$, and $\sZ$ be smooth, proper curves over $S$ whose generic fibers are all geometrically connected and have genus at least 2. Let \[\xymatrix{ & \sZ\ar[dl]_f \ar[dr]^g & \\ \sX & & \sY}\]be a finite \'etale correspondence of schemes. Then the base-changed correspondence of curves over $\eta$ has a core if and only if the correspondence of curves over $s$ has a core.
		
	\end{lemma}
	
	\begin{proof}This is essentially contained in \cite[Lemma 4.10, 4.14]{Kr}, but we present a slightly different proof of the ``not having a core specializes" direction. First of all, having a core is invariant under field extension; therefore, we may assume that $s$ is the spectrum of an algebraically closed field. Suppose $X_{s}\leftarrow Z_{s}\rightarrow Y_{s}$ has a core. Then there exists a connected curve $W_s$ over $s$, equipped with a map $W_s\rightarrow Z_s$, such that under the induced maps $W_s$ is finite Galois over $X_s$ and $Y_s$. As finite \'etale covers canonically lift over finite-order thickenings of the base, $W_s$ canonically deforms to a formal scheme $\tilde{W}$ over $Spf(R)$. As $W_s$ is a smooth projective curve, the formal scheme $\tilde{W}$ canonically algebraizes to a smooth projective curve $\sW$ over $S$ by Grothendieck's formal existence theorem. Therefore by Lemma \ref{finitegaloiscore}, the correspondence of curves over $\eta$ has a core.
		
		For the other direction, that ``having a core specializes" see \cite[Lemma 4.14]{Kr}.
		
	\end{proof}

	We construct certain Hecke correspondences for Shimura curves. Then using Theorem \ref{oneclump} and work of Wortmann on the $\mu$-ordinary locus of Hodge-type Shimura varieties, we conclude that there are exactly two Newton strata. In fact, our techniques show there are exactly two central leaves.

	\begin{definition}\label{hecke}Let $(G,X)$ denote a Shimura curve datum, given by $F$ and $D$. Let $K\subset G(\A_f)$ be a sufficiently small open compact subgroup such that the associated Shimura curve $Sh_K(G,X)$ is a scheme. Let $l$ be primes of $\Q$ and $\lambda | l$ be a prime of $F$ such that
		
		\begin{itemize}
			
			\item $l$ is unramified in $F$
			
			\item The prime $\lambda$ has inertial degree 1 over $l$
			
			\item $D$ is split at $\lambda$
			
			\item $K$ is hyperspecial at $l$.
			
		\end{itemize}We define the correspondence of curves, $T_{\lambda}$, as the Hecke correspondence associated to $g\in G(\A_f)$ given by $(1,\dots,1,g_\lambda,1,\dots)$ where $g_\lambda$ is the matrix $\left(\begin{matrix}l & 0\\
			
			0 & 1
			
		\end{matrix}\right)$ under a choice of isomorphism $D\otimes F_{\lambda}\cong M_{2\times 2}(\Q_l)$.
		
	\end{definition}
	
	\begin{lemma}\label{tlambda_nocore}Let $(G,X)$ be a Shimura curve datum and let $K\subset \A_f$ be a sufficiently small open compact subgroup so that $Sh_K(G,X)$ is a scheme. Then
		
		\begin{enumerate}
			
			\item $T_{\lambda}$ is a finite \'etale correspondence of compact complex curves that descends to $F$, the reflex field of $(G,X)$, over which the canonical model of $M$ is defined.
			
			\item When $T_\lambda$ is restricted to geometrically connected components, it is an \'etale correspondence of curves without a core of bidegree $(l+1,l+1)$.
			
			\item If $K$ is hyperspecial at $p$, and $l\neq p$, then $T_{\lambda}$ has good reduction modulo $p$. When restricted to any geometrically connected component modulo $p$, it yields an \'etale correspondence of curves without a core.
			
		\end{enumerate}
		
	\end{lemma}
	
	\begin{proof}
		
		By construction, $T_\lambda$ is a finite \'etale correspondence of (possibly disconnected) compact Riemann surfaces. We first prove that the correspondence, when restricted to (complex) connected components, has no core.

		As we are studying the connected components, we work with the derived group $G^{der}$. Write the correspondence of geometrically connected curves as $M\xleftarrow{f} Z\xrightarrow{g} M$. By the assumptions on $l$ and $\lambda$, $G^{der}(\Z_l)\cong \SL_2(\Z_l)\times H$, where $H$ is an $l$-adic group. Then the correspondence $T_{\lambda}$ ``comes from" the pair of inclusions of $l$-adic groups \[\xymatrix{ & \Gamma_0(l)\ar[dl] \ar[dr] & \\ \SL_2(\Z_l) & & \SL_2(\Z_l)}\]Here, $\Gamma_0(l)$ is the subgroup of $\SL_2(\Z_l)$ that reduces to the standard Borel mod $l$, the left hand inclusion is the standard inclusion, and the right hand inclusion is given by conjugation by the element $\left(\begin{matrix}l & 0\\
			
			0 & 1
			
		\end{matrix}\right)$.
		
		By Lemma \ref{finitegaloiscore}, to prove that $M\leftarrow Z\rightarrow M$ has no core, it suffices to prove that there is no finite \'etale cover $h:W\rightarrow Z$ such that $f\circ h$ and $g\circ h$ are both Galois. This is equivalent to proving that there is no finite index subgroup $\Gamma'\subset \Gamma_0(l)$ that is normal under both of the above embeddings into $\SL_2(\Z_l)$, a question in pure group theory.

		Fix $V=\Q_l^{\oplus 2}$ and consider the set $\sL$ of lattices (a.k.a. free rank 2 $\Z_l$-submodules) up to homothety in $V$. This discrete set can be enchanced into a tree, the \emph{Bruhat-Tits tree} $\sT$, as follows. Given two lattices $L_1$ and $L_2$, we join their representatives in $\sL$ by an edge if $L_1$ is a sublattice of index $l$ in $L_2$. (This is symmetric because the elements of $\sL$ are lattices up to homothety.) It is well-known that $\sT$ is an $l+1$-regular tree \cite[Theorem 1, p. 98-101]{Serre}. The group $\SL_2(\Q_l)$ acts on $\sT$ and the stabilizer of a vertex is isomorphic to $\SL_2(\Z_l)$. Moreover, the stabilizer of an edge $e$ between $v$ and $w$ is $\Gamma_0(l)$ \cite[p. 106-107]{Serre}. and in fact $Stab(v)\leftarrow Stab(e)\rightarrow Stab(w)$ is isomorphic to $\SL_2(\Z_l)\leftarrow \Gamma_0(l)\rightarrow \SL_2(\Z_l)$ (see \cite[Corollaire 1, p. 110]{Serre} and its following comments).

		Now we claim that there is no finite index subgroup $\Gamma'$ of $\Gamma_0(l)$ that is normal in both copies of $\SL_2(\Z_l)$. Suppose that $\Gamma'$ is a normal subgroup of $Stab(v)\cong \SL_2(\Z_l)$. Let $\sT_{\Gamma'}$ be the subgraph of $\sT$ fixed by $\Gamma'$. The ends emanating from $v$ are in bijective correspondence with $\P V$ in an $\SL_2(\Z_l)$-equivariant way \cite[P. 101]{Serre}. Hence $\SL_2(\Z_l)$ acts transitively on the ends emanating from $v$. As $\Gamma'\subset \SL_2(\Z_l)$ is normal and of finite index, the subgraph $\sT_{\Gamma'}$ is a ball around $v$ of some \emph{finite} radius. On the other hand, there is no \emph{finite} subgraph of $\sT$ that is a ball around both $v$ and $w$; therefore there is no finite index subgroup $\Gamma'\subset \Gamma_0(l)$ that is normal under both of the embeddings into $\SL_2(\Z_l)$. Translating back to algebraic geometry, there is no finite \'etale cover of $Z$ that is Galois over both copies of $M$. By Lemma \ref{finitegaloiscore}, this finite \'etale correspondence of curves over $\C$ has no core. The (bi)degree is just the index of a Borel in $\SL_2(\F_l)$; by considering the action on $\P^1_{\F_l}$, this is $l+1$.

		By the discussion on canonical models in Appendix \ref{shimura_varieties}, $T_\lambda$ descends to the reflex field of $(G,X)$ because the construction of canonical models is Hecke-equivariant. When $K$ is hyperspecial at $p$ and $l\neq p$, then by the discussion on integral canonical models, for each prime $\pP|p$, the Hecke correspondence has a model over $\sO_{F,\pP}$. We may take an unramified extension to pick out the geometrically connected components by \cite[2.2.4]{Ki}; then the special fibers are also geometrically connected by Zariski's principle of connectedness. By applying Lemma \ref{nocorespecializes}, we see that the reduction modulo $p$ of the Hecke correspondence of geometrically connected components has no core.

	\end{proof}
	
	
	%

	For each geometric point $x\in \sM(\overline{\F}_p)$, we may consider the associated abelian variety $\sU_x$. It follows from a theorem of Grothendieck-Katz that the Newton polygon jumps on a Zariski closed subset of $\sM(\overline{\F}_p)$. For each geometrically connected component $\sM_i$ of $\sM$, we have an abelian scheme $\sU_i\rightarrow \sM_i$. Over $\bar \F_p$, these connected components are all isomorphic. Therefore there are only finitely many points of $\sM(\bar \F_p)$ where the Newton polygon jumps.
	
	\begin{definition}Notation as in \ref{shimura_curve}. Then the \emph{Newton jumping locus} $\sN_p$ is the subset of points $x\in \sM(\overline{\F}_p)$ such that $\sU_x$ has a Newton polygon which is higher than the generic Newton polygon. 
	\end{definition}
	
	\begin{corollary}\label{independence}
		
		Notations as in \ref{shimura_curve}. Then the Newton jumping locus $\sN_p$ consists of only one Newton polygon. The set $\sN_p\subset \sM(\overline{\F}_p)$ is independent of the rational symplectic representation realizing $(G,X)$ as a Hodge-type Shimura curve.
		
	\end{corollary}
	
	\begin{proof}
		
		We first argue that $\sN_p$ intersects every geometrically connected component non-trivially. By \cite[Theorem 1.1]{Wort}, the $\mu$-ordinary locus of $\sM$ is open and dense. Any pair of points $x\in\sM(\overline{\F}_p)\backslash \sN_p$ are Newton generic and hence $\mu$-ordinary. Then \cite[Theorems 1.2]{Wort} implies thats $\sU_{x}[p^\infty]$ and $\sU_{y}[p^\infty]$ are isomorphic as $p$-divisible groups over $\overline{\F}_p$. If $\sN_p$ were empty, then in particular $\sM$ would live entirely in a fixed Ekedahl-Oort strata of $\scrA_{n,1,K'}$. On the other hand, such EO strata are quasi-affine by \cite[Theorem 1.2]{Oort}. As each $\sM_i$ is proper and connected, this proves that the map is constant, which is a contradiction. In particular, the Newton polygon jumps on every (geometrically) connected component. Now, as the $\sU_i\rightarrow \sM_i$ are all isomorphic over $\bar{\F}_p$, the special Newton polygons are the same on the different geometrically connected components. Note that it still, \`a priori, depends on the choice of rational symplectic embedding. We now focus on a fixed geometrically-connected component, $\sM_0$.

		Pick a $\lambda$ in $F$, prime to $p$, such that the conditions in Definition \ref{hecke} are met. Then $T_{\lambda}$ restricts to an \'etale correspondence without a core 
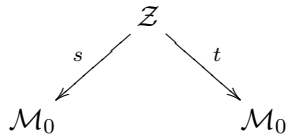
\[\xymatrix{ & \sZ\ar[dl]_s \ar[dr]^t & \\ \sM_0 & & \sM_0}\] over $\overline{\F}_p$ by Lemma \ref{tlambda_nocore}.

		Recall that $s^*\sU_0$ is isogenous to $t^*\sU_0$ as abelian schemes over $\sX$ by Lemma \ref{heckepullbackisogenous}. Therefore if $\sN_p$ consisted of more than one type of Newton polygon, then there would be at least two clumps. On the other hand, Theorem \ref{oneclump} implies that there is at most one clump.

		Now, the set $\sN_p\subset\sM(\overline{\F}_p)$ does not depend on the choice of symplectic representation: indeed, the Hecke correspondence $\sM_0\leftarrow\sX\rightarrow\sM_0$ is \emph{independent} of the choice of symplectic representation by \cite[2.3.8 (ii)]{Ki}, $\sN_p$ is the unique clump of this Hecke correspondence, and the abelian schemes $\sU_i\rightarrow \sM_i$ over each geometrically connected components are all isomorphic.
		
	\end{proof}
	
	\begin{remark} In fact, $s^*\sU$ is separably isogenous to $t^*\sU$ by the construction of the Hecke correspondence $T_\lambda$; therefore, for $x,y\in\sN_p$, the $p$-divisible groups $\sU_{x}[p^\infty]$ and $\sU_{y}[p^\infty]$ are isomorphic as $p$-divisible groups over $\overline{\F}_p$ by Theorem \ref{oneclump}.
		
	\end{remark}

	\begin{proposition}\label{HW-map}
		
		Notation as above. Then the Hasse-Witt map $HW(\sM_0)$ has the following properties:
		
		\begin{itemize}
			
			\item [(i)] $HW(\sM_0)\neq 0$;
			
			\item [(ii)] $\mathrm{div}(HW(\sM_0))$ is multiplicity free;
			
			\item [(iii)] The Newton jumping locus of $\sU_0\rightarrow \sM_0$ is $\mathrm{div}(HW(\sM_0))$
			
			
			
		\end{itemize}
		
	\end{proposition}

	\begin{proof}
		
		Part (i) follows from \cite[Remark 3.4]{SZZ} by Honda-Tate theory. Part (ii) is given by \cite[Proposition 5.3]{SZZ} with an explicit calculation using displays. (Both computations rely on a fixed mod\`ele \'etrange.) Finally, consider the Hecke correspondence $T_{\lambda}$: \[\xymatrix{ & \sZ_0\ar[dl]_s \ar[dr]^t & \\ \sM_0 & & \sM_0}\]
		By Lemma \ref{tlambda_nocore}, $T_\lambda$ is an \'etale correspondence without a core. Note that $$
		s^*(E_{unif},\theta_{unif})\cong t^*(E_{unif},\theta_{unif})
		$$
		because both $s$ and $t$ are finite \'etale; moreover, this pulled-back Higgs bundle is uniformizing on $Z$. As $(E_{unif},\theta_{unif})$ is $f$-periodic, so is the pulled-back Higgs bundle.\footnote{This follows from the fact that the inverse Cartier commutes with finite \'etale pullback, see e.g. \cite[Theorem 5.3]{Lan19}.} From the construction of the Hasse-Witt invariant, it follows that $s^*(HW(\sM_0))$ and $t^*(HW(\sM_0))$ are both proportional to $HW(\sZ_0)$.
		
		Therefore the Hasse-Witt invariant yields a nonzero invariant pluricanonical differential form on $T_{\lambda}$. By taking the zero locus, we obtain that $\mathrm{div}(HW)$, considered as a finite set of points, is a clump. As the Newton jumping locus is non-empty (see the proof of Corollary \ref{independence}), it also yields a clump on $T_\lambda$. By Theorem \ref{oneclump}, these two clumps must coincide; hence $\sN_p=\mathrm{div}(HW(\sM_0))$ as desired.
		
		
	\end{proof}
	The mass formula immediately follows.
	
	\begin{corollary}\label{massformula}
		
		Notation as in Theorem \ref{main_result}. Then the following mass formula holds:
		
		$$|\sN_p|=(p^f-1)(g-1),$$
		
		where $f=[F_{\mathfrak p}: \Q_p]$ and $g$ is the genus of the curve $X$.
		
	\end{corollary}
	
	\begin{proof}
		
		The Newton jumping locus in each geometrically connected component has the same cardinality. Hence, we reduce to a geometrically connected component. By Proposition \ref{HW-map}, this is the degree of $$HW(\sM_0)\in H^0(\sM_0,K ^{\otimes\frac{p^f-1}{2}}_{\sM_0}).$$
		
		This degree is $(p^f-1)(g-1)$, as desired.
		
	\end{proof}
	
	\begin{proof}[Proof of Theorem \ref{main_result}]
		
		Combine Corollary \ref{independence} and Corollary \ref{massformula}.
		
	\end{proof}
	
	\section{Periodic Higgs conjecture}
	Let $C$ be a smooth projective curve over $\C$. Recall from \cite[\S4, \S6]{KS} that we introduced the category $\PDR(C)$ (resp. $\PHG(C)$) of periodic de Rham (resp. Higgs) bundles over $C$, as well as the full subcategory $\MDR(C)$ (resp. $\MHG(C)$) of motivic de Rham (resp. Higgs) bundles. They sit in the four corners of the following commutative diagram:
	\begin{equation*}\label{eq1}
		\begin{CD}
			\MDR(C)@>\iota>>\PDR(C)\\
			@V\Gr VV@  VV\Gr V\\
			\MHG(C)@>\iota>>\PHG(C).
		\end{CD}
	\end{equation*}
	The left vertical arrow is an equivalence of semisimple categories over $\C$ (\cite[Lemmata 5.2 5.4]{KS}). Recall that the periodic de Rham conjecture \cite[Conjecture 1.4]{KS} predicts that the top horizontal arrow is an equivalence (of semisimple categories over $\C$). As a Higgs companion, we propose the following
	\begin{conjecture}[Periodic Higgs conjecture]\label{motivicity conj}
		Notations as above. Then the inclusion functor
		$$
		\iota: \MHG(C)\to \PHG(C)
		$$
		is essentially surjective, viz. an equivalence of categories.
	\end{conjecture}
	Since every object in $\MHG(C)$ is semisimple, the polystability conjecture follows from the above conjecture. By \cite[Lemma 5.5 (iii), Lemma 6.1]{KS}, the polystability conjecture alone implies that the right vertical arrow is fully faithful.

	We note that the periodic Higgs conjecture is indeed a strong conjecture: First, it implies the periodic de Rham conjecture (for $C$). Second, it implies the following subsidiary conjecture:
	\begin{conjecture}[Arithmetic Simpson correspondence]
		Let $C$ be a smooth projective curve over $\C$. The grading functor
		$$
		\Gr: \PDR(C)\to \PHG(C)
		$$
		is an equivalence of semsimple categories.
	\end{conjecture}
	By Proposition \ref{periodic_line_bundles}, the periodic Higgs conjecture holds for rank one case. This has the following consequence.
	\begin{corollary}
		Let $U$ be a smooth curve defined over $\bar \Q$, and $(V,\nabla)$ be a rank one connection over $U/\bar \Q$. Suppose $(V,\nabla,Fil_{tr})$ is periodic with bounded period for $p$ running over a set of (Dirichlet) density one primes. Then $(V,\nabla)$ is motivic and hence the complex local system over $U^{an}$ corresponding to $(V,\nabla)$ has finite monodromy.
	\end{corollary}
	Note that the statement does not follow directly from the solution of Grothendieck-Katz $p$-curvature conjecture in the case of rank one connections: the $p$-curvature conjecture requires a condition at \emph{all but finitely many primes}, not merely a Dirichlet density 1 set of primes. Under the assumption of the $p$-curvature conjecture, one knows that $(V,\nabla)$ modulo $\mathfrak{p}$ is torsion for all $\mathfrak{p}\gg 0$; however, a priori one has no control over the torsion order. Under our hypothesis, we only require a periodicity at a density 1 set of primes; this will force the torsion order of $(V,\nabla)$ modulo $\mathfrak p$ to be more well-behaved. We further note that this gives an alternative proof of \cite[Proposition 6.7]{KS}.
	\begin{proof}
		Let $C$ be the compactification of $U$, which is also defined over $\bar \Q$. By the celebrated monodromy theorem of N. Katz (see e.g. \cite[Theorem 8.1 (3)]{Kat82}), the local monodromy around each point of $C-U$ is of finite order. Consequently, after taking a finite \'etale cover $\pi: U'\to U$ over $\bar \Q$, the pullback connection $(V',\nabla')$ over $U'$ extends to a connection over $C'$, the compactification of $U'$. So we are reduced to the case when $U=C$ is projective. By this case, we notice that in the proof of Proposition \ref{periodic_line_bundles}, the celebrated theorem of R. Pink \cite{Pink} requires \emph{only} the density-one-primes condition. It follows that every rank one Higgs bundle over $C$ which is periodic for a set of density one primes is torsion, and hence motivic. If one changes the almost-all-primes condition in the definition of a periodic object over $C'$ by the density-one-primes condition, we shall obtain the categories $\PDR^w(C'), \PHG^w(C')$ of periodic objects in this new sense, as well as their full motivic subcategories $\MDR^w(C'), \MHG^w(C')$. We still obtain a commutative diagram of functors:
		\begin{equation*}\label{eq1}
			\begin{CD}
				\MDR_1^w(C')@>\iota>>\PDR_1^w(C')\\
				@V\Gr VV@  VV\Gr V\\
				\MHG_1^w(C')@>\iota>>\PHG_1^w(C'),
			\end{CD}
		\end{equation*}
		where the subscript 1 means a full subcategory consisting of rank one objects. Since the left vertical arrow is essentially surjective by definition and the bottom horizontal arrow is an equivalence of categories, the top horizontal arrow must be essentially surjective and hence an equivalence of categories too. (Actually, all arrows are equivalence of categories.) Since the complex local system associated to a rank one motivic de Rham bundle over $C'$ is of finite image, the statement follows.
	 \end{proof}

	\appendix
	\section{Quaternionic Shimura varieties and Deligne's mod\`ele \'etrange}\label{shimura_varieties}
	
	In this appendix, we collect necessary results on Shimura varieties, integral models, and Hecke correspondences. Our main sources are Carayol \cite{Car}, Deligne \cite{De72}, \cite{De79}, Kisin \cite{Ki}, Milne \cite{Milne} and Reimann \cite{R}. We explicate the important fact that the ``polydisk Shimura modular stacks" of \cite[\S 9]{CS} are indeed examples of Deligne's mod\`ele \'etrange.

	Let $F$ be a totally real field of degree $d$ over $\Q$ and let $D$ be a quaternion algebra over $F$ that is split at $r\geq 1$ real places of $F$. Labelling the places $\tau_1,\dots,\tau_d$, we suppose the first $r$ split $D$ and we fix splittings $s_i: D\otimes_{F,\tau_i}\R\cong M_{2\times 2}(\R)$ for $1\leq i \leq r$.

	\begin{definition}\label{arithmetic lattice}Pick a maximal order $\sO_D\subset D$. Via the splittings we may embed $$\Gamma:=\sO^{N=1}_D\to \SL_2(\R)^r,$$
		where $N=1$ signifies ``norm=1". For a finite index subgroup $\Gamma'\subset \Gamma$, the quotient $[X_{\Gamma'}]:=[\Gamma'\backslash(\HH)^r]$ is the analytification of a connected complex Deligne-Mumford (DM) stack. We call $[X_{\Gamma'}]$ a \emph{connected quaternionic Shimura stack}, and $X_{\Gamma'}$ a \emph{connected quaternionic Shimura variety}. 
	\end{definition}
	

	\begin{definition}\label{def:quaternionic_shimura}Consider $G:=\textrm{Res}_{F/\Q}(D^\times)$ as an algebraic group over $\Q$. We get $G_{\R}$-conjugacy class $X$ of cocharacters $$\SSS\rightarrow G_{\R}\cong \GL_2(\R)^r\times \textrm{U}(2)^{d-r}$$ by sending an element $z=x+\sqrt{-1}y \in \SSS(\R)=\C^*$ to
		$$(\begin{pmatrix}x & y\\ -y&x\end{pmatrix},\dots,\begin{pmatrix}x & y\\ -y&x\end{pmatrix},1,\cdots,1).$$
		Then $(G,X)$ is a Shimura datum. We have the explicit description: $X=(\P^1(\C)\backslash\P^1(\R))^r$.
		Let $\A_f$ denote the finite ad\`eles over $\Q$. We obtain an action of $G(\Q)$ on $X $ via the embedding $G(\Q)\hookrightarrow G(\R)$. Then the \emph{quaternionic Shimura variety} associated to $(G,X)$ is the following quotient:
		$$Sh(G,X):=G(\Q)\backslash G(\A_f)\times X.$$
		This is the inverse limit of ad\`elic double quotients
		$$Sh_K(G,X):=G(\Q)\backslash G(\A_f)\times X/K,$$
		where $K$ runs through sufficiently small (i.e., neat) open compact subgroups of $G(\A_f)$ \cite[Theorem 5.28]{Milne}.
	\end{definition} 
	In the general formulation, potential disconectedness is built in: $Sh_K(G,X)$ is a finite union of connected Shimura varieties. If $T:=\textrm{Res}_{F/\Q}\mathbb{G}_m$, then there is a map $\nu\colon T\rightarrow \mathbb{G}_m$ given by ``reduced norm" (with kernel $G_1$, the derived group of $G$).

	If $D/F$ is not globally split and $d>1$, the Shimura variety $Sh(G,X)$ is not naturally a moduli space of abelian varieties. However, by using Deligne's \emph{mod\`eles \'etranges}, we can construct (non-canonically) a different Shimura variety with isomorphic geometrically connected components that does underlie a ``modular" PEL-type family of abelian varieties (\S6 \cite{De72}, \S2.7 \cite{De79}, \S2 \cite{Car}, \S1-2 \cite{R}). We follow \cite[\S 2]{Car} closely.

	Choose a totally negative element $\lambda\in F$ and an $\alpha=\sqrt{\lambda}$. Set $E:=F(\alpha)$ with involution $z\mapsto \bar z$. Following Carayol, we always consider $E$ as a fixed subfield of $\C$. Let $B:=D\otimes_F E$. Let $T_E:=\textrm{Res}_{E/\Q}\mathbb G_m$ and let $U_E\subset T_E$ be the subgroup determined by $z\bar z=1$. Let $G'':=G*_{Z(G)} T_E$ be the \emph{amalgamted product} (a.k.a. coproduct).

	Consider the homomorphism $T_E\rightarrow T\times U_E$ given as $z\mapsto (z\bar z, z/{\bar z})$. This agrees on $Z(G)\cong T$ with the homomorphism $G\rightarrow T\times U_E$ given as $g\mapsto (\nu(g),1)$. Hence there is an induced homomorphism
	$$G''\xrightarrow{\nu'} T\times U_E,$$
	defined by $(g,z)\mapsto (\nu(g)z\bar z,z/\bar{z})$. Consider the subtorus $T':=\mathbb{G}_m\times U_E\subset T\times U_E$ and let $G':=(\nu')^{-1}(T')$. Then the derived groups of $G$ and $G'$ are isomorphic. We now follow \cite[2.13]{Car} to construct $X'$ so that $(G',X')$ is a Shimura datum.
	
	Abusing notation, we choose complex embeddings $\tau_1,\dots,\tau_d\colon E\rightarrow \C$ which restrict to the ones of $F$, which yield an isomorphism
	$$T_E(\R)=(E\otimes_{\Q}\R)^{\times}\cong (\C^{\times})^d.$$
	Let $h_E\colon \SSS\rightarrow (T_E)_\R$ be the map $z\mapsto (1,\dots,1,z,\dots,z)$ where the first $r$ entries are ``1". Then the projection $h\times h_E\colon \SSS\rightarrow (G*_T T_E)_\R$ factors through a map $h'\colon \SSS\rightarrow G'_{\R}$. The $G'(\R)$-conjugacy class of $h'$ is denoted by $X'$. In particular, we obtain a Shimura datum $(G',X')$.
	
	We now explain how to realize this Shimura datum as being of PEL type \cite[\S 2.2]{Car}. Denote by $b\mapsto \bar b$ the anti-involution on $B$ which is tensor product of the canonical anti-involution on $D/F$ and the complex conjugation on $E$. Here, we follow the notation of Caroyol. Later, to avoid confusion with other involutions, we name the above anti-involution $\text{inv}_1$. Let $\delta\in B$ be a symmetric element: $\delta = \bar \delta$. Then we may construct a new anti-involution on $B$:
	$$b\mapsto b^*:=\delta^{-1}\bar b \delta.$$
	Later, we refer to this anti-involution as $\text{inv}_2$, but for the following we stick to Caroyol's notation. Let $V$ be the vector space underlying $B$. Then we have the following forms on $V$ (where $\textrm{tr}_{B/E}$ denotes the reduced trace):
	\begin{align*}
		\displaystyle \psi_L(v,w)&:=\alpha \textrm{tr}_{B/E}(v\delta w^*),\\
		\displaystyle \psi(v,w)&:=\textrm{Tr}_{E/\Q}(\alpha \textrm{tr}_{B/E}(v\delta w^*)).
	\end{align*}
	Here, $\psi$ is symplectic and non-degenerate. One checks that $G'$ is isomorphic to the group of $B$-linear symplectic similtudes $GSP(V,\psi)$, where the action is given as $b.v:=vb^*$. In particular, for any small open compact $K'\subset G'(\A_f)$, the Shimura variety $Sh_{K'}(G',X')$ carries a ``universal" family of abelian varieties of dimension $4d$ with multiplication some order $\sO\subset B$.
	
	Now, we claim that the Shimura varities we thus obtain include those constructed by Corlette-Simpson in \cite[\S 9]{CS}. We quickly review their formulation (and we use largely their notations for ease of translation).
	
	Let $L$ be a CM field, with ring of integers $\sO_L$, let $P$ be a rank 2 projective $\sO_L$-module, and let $\Phi\colon P\otimes P\rightarrow \sO_L$ be an anti-Hermitian form with respect to the natural complex conjugation $\bar{\ }\colon L\rightarrow L$. (Corlette-Simpson use $\iota$ to denote complex conjugation.) Set $\mathcal U(P,\Phi)$ to be the group of $L$-linear transformations of $P$ that preserve $\Phi$. This is a lattice in the group $\mathcal U(P_L,\Phi)$. 
	
	Corlette-Simpson define the notion of a \emph{variation of Hodge structure of type $(P,\Phi)$ on a smooth variety} on p. 1318 of \emph{loc. cit.} Set $\mathcal D(P,\Phi)$ to be the associated period domain, which is isomorphic to $\HH^r$, where $r$ is half the number of mixed embeddings for $\Phi$. Then, Theorem 9.2 of \emph{loc. cit.} states that the moduli problem of VHS of type $(P,\Phi)$ is represented by a unique smooth DM stack, the analytification of which is isomorphic to $\mathfrak{H}^{an}(P,\Phi):=\mathcal D(P,\Phi)/\mathcal U(P,\Phi)$.  Let $\sV$ be the tautological rank two local system of type $(P,\Phi)$ over $\mathfrak{H}^{an}(P,\Phi)$.
	\begin{lemma}\label{polydisk}
		Let $(L,P,\Phi)$ be as above and assume $r>0$ (that is, there is at least one mixed embedding for $\Phi$).  Let $F\subset L$ be the totally real subfield of index two. Then there exists a quaternion algebra $D/F$, split exactly over $r$ real places of $F$, and a choice of $\alpha$ and $\delta$ such that the group of $B:=D\otimes_FL$-linear symplectic similtudes $GSP(P_L,\psi)$ is isomorphic to $\mathcal{U}(P_L,\Phi)$. Therefore, if $X$ is a polydisk Shimura modular DM stack of Corlette-Simpson, then it is a connected quaternionic Shimura stack. Moreover, for each mixed embedding $\sigma: L\to \C$, there is a unique $1\leq i\leq r$ such that $\sV\otimes_{L,\sigma}\C\cong \L_i\otimes _{\R}\C$, where $\{\L_i\}_{1\leq i\leq r}$ are tautological rank two $\R$-local systems over the connected quaternionic Shimura stack, as defined in the proof of Lemma \ref{modular_characterization}.
	\end{lemma}
	\begin{proof}
		Write $F=L(\sqrt{a})$, where $a\in F$ is a totally negative element. Set $\alpha = \sqrt{a}\in L$. Set $\tilde{\Phi}:=\alpha \Phi$. Then $\tilde{\Phi}$ is a Hermitian form. We define the algebra $D$ as follows:
		$$D:=\{A\in\text{End}_L(P_L)| \tilde{\Phi}(Au,v)=\tilde{\Phi}(u,Av), \ \forall u,v\in P_L\}.$$
		Then $D$ is an $F$-subalgebra of $\text{End}_L(P_L)$. Then $D$ is a quaternion algebra over $F$. Indeed, it is easy to see that $\dim_F(D)=4$ and that $F$ is in the center of $D$. Moreover, it is simple because $D\otimes_F L\cong \text{End}_L(P_L)$ is simple.  
		
		Fix an isomorphism $\text{End}_L(P_L)\cong M_{2\times 2}(L)$ which induces $s: D\otimes_F L\cong M_{2\times 2}(L)$. Recall that we had above defined an anti-involution on $D\otimes_F L$, given by the tensor product of the natural anti-involution on $D$ and complex conjugation on $L$. To avoid confusion, we write this anti-involution as $\text{inv}_1\colon D\otimes_F L \rightarrow D\otimes_F L$. Next, we claim that $\text{inv}_1$ may be conjugated to be the usual conjugate-transpose map, after transport of structure.
		
		Using the fixed splitting of $D\otimes_F L$, we consider $\text{inv}_1$ as an anti-involution on $M_{2\times 2}(L)$. Then the map:
		
		$$b\longmapsto \overline{\text{inv}_1(b)}^t,$$
		obtained by composition $\text{inv}_1$ and usual ``conjugate transpose'', is an involution on $M_{2\times 2}(L)$, and is moreover $L$-linear. By the Skolem-Noether theorem, there exists $g\in \text{GL}_2(L)$ such that:
		
		$$gbg^{-1} = \overline{\text{inv}_1(b)}^t,$$
		for all $b\in M_{2\times 2}(L)$. Moreover, as $g*\text{inv}_1(b)*g^{-1} = \overline{b}^t$, we see that $\text{inv}_1(g)* b * \text{inv}_1(g)^{-1} = \text{inv}_1(\overline{b}^t)$. On the other hand, $\text{inv}_1$ and conjugate transpose commute (use the fact that a group in which every element has order 2 is commutative). This implies that $g$ is symmetric, i.e., that $\text{inv}_1(g) = g$. Set $\delta :=g$.
		
		Now that we have nailed down all of the choices, we must show that the group $G'$ of $B$-linear symplectic similtudes and $\mathcal U(P_L,\Phi)$ are isomorphic as algebraic groups. We first construct a natural map:
		
		\begin{equation}\label{eqn:CS_to_Car}\mathcal U(P_L,\Phi)\rightarrow \text{Aut}_L(V,\psi_L,B).\end{equation}
		As $V = B= D\otimes_F L\cong M_{2\times 2}(L)$, the definition of the map is straightforward: we send a matrix $A\in \mathcal U(P_L,\Phi)$ to the transform ``left multiplication by $A$'', which we denote by $L_A$. It is immediate to check that the map respects $\psi_L$ and that $L_A$ commutes with the action of $B$ on $V$. (Recall that $b.v:=vb^*=v*\text{inv}_1(b)$.)
		
		We now claim that the map in Equation \ref{eqn:CS_to_Car} is an isomorphism of groups. The map is clearly injective, so our task is to show that it is surjective. First of all, we claim that any $\varphi \in \text{Aut}_L(V,B)$ is automatically of the form $L_A$ for some $A\in \text{GL}_2(L)$. This is essentially because $B$ acts transitively on $V$. More precisely, if $\varphi \in \text{Aut}_L(V,B)$, then $$\varphi(b) = \varphi( b^*.\begin{pmatrix}1&0\\0&1\end{pmatrix}) = b^*.\varphi(\begin{pmatrix}1&0\\0&1\end{pmatrix}) = \varphi(\begin{pmatrix}1&0\\0&1\end{pmatrix})b,$$
		which implies that $\varphi = L_A$ for $A=\varphi(\begin{pmatrix}1&0\\0&1\end{pmatrix})$.
		
		Next we must check that $\varphi = L_A \in \text{Aut}_L(V,\psi_L,B)$ implies that $A\in \mathcal U(P_L,\Phi)$, i.e., that $A$ respects the anti-Hermitian form $\Phi$ (or, equivalently, the Hermitian form $\tilde{\Phi}$). This is a straightforward computation. Pick a basis of $P_L$ in which (the transport of structure of) $\text{inv}_1$ is simply conjugate transpose, and let $H$ be the matrix of of the Hermitian form $\tilde{\Phi}$ in this basis. Then the defining condition on $\varphi$ is that:
		$$\psi_L(\varphi(w'),\varphi(w)) = \psi_L(w',w),\ \forall w,w'\in V.$$
		As $\varphi = L_A$, we may expand this out into the following:
		$$\text{Tr}_{B/L}(Aw'\overline{w}^t\bar{A}^t H) = \text{Tr}_{B/L}(w'\overline{w}^tH).$$
		(Here, we simply use the definition of $\psi_L$ together with the fact that in our basis, $\text{inv}_1$ is simply conjugate transpose.) On the other hand, the left hand side is $\text{Tr}_{B/L}(w'\overline{w}^t\bar{A}^t HA)$ as the trace is invariant under cyclic permutation. By plugging in the matrices $w=w'=E_{ij}$, where $E_{ij}$ is a matrix whose only non-zero entry is 1 in position $(i,j)$, for $i,j \in \{1,2\}$, one deduces that $H = \bar{A}^tHA$, which in particular implies that $A\in \mathcal U(P_L,\Phi)$.
		
		Now, there is a natural composition: 
		$$\mathcal U(P_L,\Phi)\rightarrow \text{Aut}_L(V,\psi_L,B)\rightarrow G'= \text{Aut}_{\Q}(V,\psi,B),$$
		and we claim it is an isomorphism. First of all, if $\varphi' \in G'$, then $\varphi'$ is automatically $L$-linear because of the $B$-equivariance. (Recall that $L$ is the center of $B$.) Now, we may write $\varphi' = L_{A}$, for some $A\in \text{GL}_2(L)$. Then $L_{A}$ is in $\text{Aut}_{\Q}(V,\psi,B)$ if and only if the following equality holds for all $w,w'\in V$: 
		$$Tr_{L/\Q}(\alpha Tr_{B/L}(w'\overline{w}^t\overline{A}^tHA))=Tr_{L/\Q}(\alpha Tr_{B/L}(w'\overline{w}^tH)).$$
		
		Set $u_1=\overline{w}^t\overline{A}^tHA$ and $u_2=\overline{w}^tH$. Now, the form $Tr_{L/\Q}(\alpha Tr_{B/L}(\cdot,\cdot): V\otimes V\to \Q$ is non-degenerate. Varying $w'\in V$, it immediately follows that $u_1=u_2$. Setting $w=\begin{pmatrix}1&0\\0&1\end{pmatrix}$, we get further that $A^*HA=A$. This implies that $A\in \mathcal U(P_L,\Phi)$, exactly as desired.	
		
		Finally, let $\{\sigma_i, \overline{\sigma}_i\}_{1\leq i\leq d}$ be the set of complex embeddings of $L$ such that the first $2r$ embeddings are mixed for $\Phi$. Recall that the tautological rank two $L$-local system $\sV\otimes_{\sO_L}L$ is induced by the standard representation of $\sU(P_L,\Phi)\subset \text{End}_L(P_L)\cong M_{2\times 2}(L)$ on $L^{2}$.  Only for $1\leq i\leq r$ are $\sV\otimes_{L,\sigma_i}\C$ as well as its complex conjugation $\sV\otimes_{L,\overline{\sigma}_i}\C$ non-unitary.  Set $\tau_i$ to be the restriction of $\sigma_i$ to $F$.  Fix an isomorphism
		$$
		s_i:D\otimes_{F, \tau_i}\R\cong M_{2\times 2}(\R)\ \textrm{or}\ H, 
		$$
		depending on $\tau_i$ is a splitting real place for $D$ or not (here $H$ is the Hamilton's quaternion algebra). As
		$$
		s_i\otimes id: (D\otimes_{F, \tau_i}\R)\otimes_{\R}\C\cong M_{2\times 2}(\C),
		$$
		the standard representation of $(M_{2\times 2}(\C))^{\times}$ on $\C^2$ induces a tautological rank two $\C$-local system of $\M_i$ over $X$. Moreover, $\M_i$ is non-unitary if and only if $\tau_i$ is a splitting prime. In that case, $\M_i\cong \L_i\otimes_{\R}\C$, where $\L_i$ is defined as in Lemma \ref{modular_characterization}. Therefore,  by composing the isomorphisms $s\otimes id,s_i\otimes id$ with the following natural isomorphism
		$$
		(D\otimes_{F}L)\otimes_{L,\sigma_i}\C \cong (D\otimes_{F,\tau_i}\R)\otimes_{\R}\C,
		$$
		we obtain an isomorphism of $\C$-local systems between $\sV\otimes _{L,\sigma_i}\C$ and $\M_i$ (so is its complex conjugation $\sV\otimes _{L,\overline{\sigma}_i}\C$). It follows that exactly the first $r$ real places $\{\tau_i\}_{1\leq i\leq r}$ splits $D$, and for such an $i$, there are isomorphisms of $\C$-local systems over $X$:
		$$
		\sV\otimes _{L,\sigma_i}\C\cong  \L_i\otimes_{\R}\C \cong \sV\otimes _{L,\overline{\sigma}_i}\C.
		$$
		The lemma is proved.
	\end{proof}

	\begin{definition}\label{shimura_curve_datum}Let $F$ be a totally real number field and let $D$ be a  quaternion division algebra over $F$ that is split at exactly one real place. The \emph{Shimura curve datum} associated to $F$ and $D$ is the quaternionic Shimura datum $(G,X)$ described in Definition \ref{def:quaternionic_shimura}.
		
	\end{definition}
	
	We now discuss Hecke correspondences. Let $(G,X)$ be a Shimura datum and $g\in G(\A_f)$. Then, $K_g:=gKg^{-1}\cap K\subset G(\A_f)$ is also an open compact subgroup. The \emph{Hecke correspondence} $T_g$ associated to $g$ is \[\xymatrix{ & Sh_{K_g}(G,X)\ar[dl]_s \ar[dr]^t & \\ Sh_K(G,X) & & Sh_K(G,X)}\]
	
	Here the first map is given by the natural inclusion $K_g\subset K$ and the second map is given by the inclusion $K_g \subset g\Gamma g^{-1}\cong K$ where the second identification comes from conjugation by $g$. This is a finite \'etale correspondence of smooth, quasi-projective, \emph{not-necessarily connected} algebraic varieties, defined over $E$: see \cite[Ch. 12-14]{Milne}, and especially 13.6 of \emph{loc. cit} for the Hecke-equivariance property.
	
	\begin{remark}Let $(G,X)$ be a Shimura curve datum, given by $F$ and $D$. Then the reflex field $E(G,X)$ is isomorphic to $F$ \cite[1.1.1]{Car}.
		
	\end{remark}

	For the Siegel Shimura datum $(GSP_{2n},S^{\pm})$, Hecke correspondences have a pleasant property. First we set some notation: if $K\subset GSP_{2n}(\A_f)$ is an open compact subgroup, we write $\scrA_{n,1,K}$ for the Shimura variety $Sh_K(GSP_{2n},S^{\pm})$; intuitively, it is the \emph{moduli space of principally polarized abelian varieties with $K$-level structure}. Let $g\in GSP_{2n}(\A_f)$. Consider the Hecke correspondence \[\xymatrix{ & \scrA_{n,1,K_g}\ar[dl]_s \ar[dr]^t & \\ \scrA_{n,1,K} & & \scrA_{n,1,K}}\]Let $\scrU\rightarrow \scrA_{n,1,K}$ be the universal family of abelian varieties over $\scrA_{n,1,K}$. Then $s^*\scrU$ is isogenous to $t^*\scrU$.

	Let $(G,X)$ denote a Shimura curve datum, let $p$ be a prime of $\Q$, and let $\mathfrak p|p$ be a prime of $F$. Suppose $K\subset G(\A_f)$ is hyperspecial at $p$ and neat. Then it follows from \cite[\S 6.1]{Car} that $Sh_K(G,X)$ has a canonical integral model $\scrS_K(G,X)$ over $\sO_{F,\mathfrak p}$. Now let $(G,X)\hookrightarrow (GSP_{2n},S^{\pm})$ realize $(G,X)$ as being of Hodge type. Then by \cite[2.1.2]{Ki} there exists $K'\subset GSP_{2n}(\A_f)$, hyperspecial at $p$, and an induced map of integral canonical models: $$\scrS_K(G,X)\rightarrow\scrA_{n,1,K'}$$
	
	which are smooth schemes over $\sO_{E,\pP}$. In particular, from this data we obtain a ``universal" abelian scheme over $\scrS_K(G,X)$.

	\begin{corollary}\label{heckepullbackisogenous}Let $(G,X)\hookrightarrow (GSP_{2n},S^{\pm})$ be a Hodge-type Shimura datum with reflex field $E$. Let $p$ be a rational prime and let $\pP$ be a prime of $E$ over $p$. Let $K=K_pK^p\subset G(\A_f)$ be a sufficiently small open compact subgroup with $K_p\subset G(\Q_p)$ hyperspecial. Let $g\in G(\A^p_f)$ and consider the Hecke correspondence of integral models over $\sO_{E,\pP}$
		
		\[\xymatrix{ & \scrS_{K_g}(G,X)\ar[dl]_s \ar[dr]^t & \\ \scrS_{K}(G,X) & & \scrS_{K}(G,X)}\]
		
		Let $\scrU_G\rightarrow \scrS_{K}(G,X)$ be the family of abelian schemes, induced from a morphism $\scrS_K(G,X)\hookrightarrow\scrA_{n,1,K'}$ as above. Then $s^*\scrU_G$ is isogenous to $t^*\scrU_G$.
		
	\end{corollary}
	
	\begin{proof}
		
		Denote by $g'\in GSP_{2n}(\A_f)$ the image of $g$ under the embedding of Shimura data as in the hypothesis of the corollary. Then the Hecke correspondence for $(G,X)$ maps to a Hecke correspondence for $(GSP_{2n},S^{\pm})$: \[\xymatrix{ & \scrA_{n,1,K'_{g'}}\ar[dl]_{s'} \ar[dr]^{t'} & \\ \scrA_{n,1,K'} & & \scrA_{n,1,K'}}\]
		
		Denote again by $\scrU\rightarrow \scrA_{n,1,K'}$ a universal family of abelian varieties over $\scrA_{n,1,K'}$. Then the fact that $s'^*\scrU$ is isogenous to $t'^*\scrU$ implies that $s^*\scrU_G$ is isogenous to $t^*\scrU_G$ as desired.
		
	\end{proof}

	\begin{notation}\label{shimura_curve}Let $(G,X)$ be a Shimura curve datum ,let $K\subset G(\A_f)$ be an open compact subgroup that is hyperspecial at an odd $p$ and let $\pP|p$ be a prime of $F$.
		
		\begin{itemize}
			
			\item Let $M$ be the Shimura curve $Sh_K(G,X)$ over $F$ (or $\C$).
			
			\item Let $\scrM$ be the canonical integral model $\scrS_K(G,X)$ over $\sO_{F,\mathfrak p}$
			
			\item Let $\sM:=\scrM\otimes\F_p$ and $\sM_0$ be a geometrically connected component of $\sM$.
			
		\end{itemize}
		
		Given a rational symplectic representation $G\rightarrow GSP_{2n}$ realizing $Sh(G,X)$ as of Hodge-type, pick $K'\subset GSP_{2n}(\A_f)$ such that we obtain $\scrM\hookrightarrow \scrA_{n,1,K'}$.
		
		\begin{itemize}
			
			\item Denote by $\scrU\rightarrow\scrM$ the induced abelian scheme over $\scrM$
			
			\item Denote by $\sU\rightarrow \sM$ the reduction modulo $p$ of $\scrU\rightarrow\scrM$.
			
			\item Denote by $\sU_0\rightarrow \sM_0$ the restriction of $\sU\rightarrow \sM$ to $\sM_0$.
			
		\end{itemize}

	\end{notation}

	\bibliographystyle{plain}

\begin{thebibliography}{5}
		
		
		
		
		\bibitem[A]{Ati} M. Atiyah, Vector bundles over an elliptic curve, Proc. Lond. Math. Soc., III. Ser. 7, 414-472 (1957).
		
		\bibitem[Ca]{Car} P. Carayol, Sur la mauvaise r\'{e}duction des courbes de Shimura, Comp. Math. 59 (1986), no. 2, 151-230.
		
		\bibitem[CC]{CC90} D. Chudnovsky, G. Chudnovsky, Computer algebra in the service of mathematical physics and number theory, Computers in mathematics, Lecture Notes in Pure and Appl. Math. 125, 109-232, 1990.
		
		\bibitem[CW]{Wolf} P. Cohen, J. Wolfart, Modular embeddings for some nonarithmetic Fuchsian groups, Acta Arith. 56 (1990), no. 2, 93-110.
		
		\bibitem[CS]{CS} K. Corlette, C. Simpson, On the classification of rank-two representations of quasiprojective fundamental groups, Compos. Math. 144 (2008), no. 5, 1271-1331.
		
		\bibitem[Del71]{De71} P. Deligne, Th\'eor\`eme de Hodge II, Publ. Math. IHES, tome 40, 5-57, 1971.
		
		
		
		\bibitem[Del72]{De72} P. Deligne, Travaux de Shimura, S\'{e}minaire. Bourbaki, Lecture Notes in Mathematics, Vol. 244, Springer, Berlin, 1972, pp. 123-165.
		
		
		\bibitem[Del79]{De79} P. Deligne, Vari\'{e}t\'{e}s de Shimura: interp\'{e}etation modulaire, et techniques de construction de mod\`eles canoniques, Automorphic Forms, Representations and L-Functions (Oregon State Univ., Corvallis, OR, 1977), Part 2, Proc. Sympos. Pure Math., vol. 33 (Amer. Math. Soc.,
		Providence, RI, 1979), 247–289.
		
		
		
		
		
		
		\bibitem[Deu]{Deuring} M. Deuring, Die Typen der Multiplkatorenringe elliptischer Funktionenkoerper, Abh. Math. Sem. Hansischen Univ. 14 (1941), 19-272.
		
		\bibitem[Elk87]{Elkies} N. Elkies, The existence of infinitely many supersingular primes for every elliptic curve over $\Q$, Invent. math. 89, 561-567 (1987).
		\bibitem[Elk89]{Elkies89} N. Elkies, Supersingular primes for elliptic curves over real number fields, Compos. Math. 72, No. 2, 165-172 (1989).
		
		
		%
		
		
		
		
		
		
		
		\bibitem[F]{Fa83}
		G. Faltings, Real projective structures on Riemann surfaces, Compositio Mathematicae, tome 48, no 2 (1983), p. 223-269.
		
		
		
		
		\bibitem[GK]{Gold}
		W. Goldring, J. Koskivirta, Strata Hasse invariants, Hecke algebras, and Galois representations, Invent. Math.  Vol. 217, 887-984, 2019.
		
		\bibitem[H]{H}
		N. Hitchin, The self-duality equations on a Riemann surface. Proc. London Math. Soc. (3), 55, 1987, no. 1. 59-126.

		\bibitem[Ka]{Kat82}
		N. Katz, A conjecture in the arithmetic theory of differential equations, Bulletin de la Soci\'et\'e Math\'ematique de France, Tome 110  (1982), p. 203-239.
		
		
		
		\bibitem[Ki]{Ki} M. Kisin, Integral models for Shimura varieties of Abelian type, J. Amer. Math. Soc 23:4 (2010), 967-1012.
		
		\bibitem[Kr]{Kr} R. Krishnamoorthy, Correspondences without a core, Algebra \& Number Theory 12:5 (2018) 1173-1214.
		
		\bibitem[KS]{KS} R. Krishnamoorthy, M. Sheng, Periodic de Rham bundles over curves, 2020.
		
		\bibitem[Ku]{Kuc} R. Kucharczyk, Modular embeddings and automorphic Higgs bundles, Algebr. Geom. 5 (2018), no. 2, 200-238.

		\bibitem[LSYZ]{LSYZ} G. Lan, M. Sheng, Y. Yang, K. Zuo, Uniformization of p-adic curves via Higgs-de Rham flows, J. Reine Angew. Math. 747 (2019), 63-108.
		
		\bibitem[LSZ15]{LSZ0} G. Lan, M. Sheng, K. Zuo, Nonabelian Hodge theory via exponential twisting, Math. Res. Lett. 22 (2015), no. 3, 859-879.
		
		\bibitem[LSZ19]{LSZ} G. Lan, M. Sheng, K. Zuo, Semistable Higgs bundles, periodic Higgs bundles, and representations of the algebraic fundamental group, J. Eur. Math. Soc. 21, no. 10 (2019), 3053-3112.
		
		
		
		
		\bibitem[Lan14]{Lan14}A. Langer, Semistable modules over Lie algebroids in positive characteristic, Documenta Mathematicae, Vol. 19 (2014) 561-592.	
		
		\bibitem[Lan19]{Lan19}A. Langer, Nearby cycles and semipositivity in positive characteristic, arXiv: 1902.05745 (2019).
		
		
		
		
		
		
		\bibitem[LS]{LS} M. Li, M. Sheng, Characterization of Beauville's algebraic numbers via Hodge theory, International Mathematics Research Notices, https://doi.org/10.1093/imrn/rnab211.
		
		
		\bibitem[Ma]{Ma89}D. Masser, Specializations of finitely generated subgroups of Abelian varieties, Transactions of the American Mathematical Society, Volume 311, Number 1, 413-424, 1989.
		
		
		
		
		
		
		\bibitem[Mi]{Milne}J. Milne, Introduction to Shimura Varieties, available at http://www.jmilne.org/math/xnotes/svi.pdf, version of September 16, 2017.
		
		
		\bibitem[Mo96]{Mo96}
		S. Mochizuki, A theory of ordinary $p$-adic curves, Publ. Res. Inst. Math. Sci. 32 (1996), no. 6, 957-1152.
		
		\bibitem[Mo98]{Mo}
		S. Mochizuki, Correspondences on hyperbolic curves, J. pure and applied algebra, 131(3), 227-244, 1998.

		\bibitem[MFK]{MFK}
		
		D. Mumford, J. Forgarty, F. Kirwan, Geometric invariant theory, Third enlarged edition, Ergebniesse der Mathematik und ihrer Grenzgebiete Volume 34, Springer.. 
	
		
		
			
		
		%

		\bibitem[O]{Oort}F. Oort, A stratification of a moduli space of abelian varieties, Moduli of abelian varieties. Birkh\"auser, Basel(2001), 345-416.
		
		
		\bibitem[OV]{OV}A. Ogus, V. Vologodsky, Nonabelian Hodge theory in characteristic $p$, Publ. Math. Inst. Hautes \'{e}tudes Sci. 106 (2007), 1-138.
		
		
		\bibitem[P]{Pink}R. Pink, On the order of the reduction of a point on an abelian variety, Math. Ann. 330 (2004), no. 2, 275-291.
		
		
		
		
		\bibitem[R]{R} H. Reimann, The Semi-Simple Zeta Function of Quaternionic Shimura Varieties, Lecture Notes in Mathematics, Vol. 1657, Springer, Berlin, 1997.
		
		
		
		
		\bibitem[SSW]{SM} P. Schmutz Schaller, J. Wolfart, Semi-arithmetic Fuchsian groups and modular embeddings, J. London Math. Soc. (2) 61 (2000), no. 1, 13-24.
		
		
		\bibitem[Se]{Serre}J.-P. Serre, H. Bass, Arbres, amalgames, $\SL_2$, Soc. math. France (1977).
		
		
		\bibitem[SZZ]{SZZ}M. Sheng, J. Zhang, K. Zuo, Higgs bundles over good reduction of Shimura curves associated with quaternion division algebra, J. reine angew. Math. 671 (2012), 223-248.
		
		
		
		\bibitem[SZ12]{SZ12} M. Sheng, K. Zuo, Polarized variation of Hodge structure of Calabi-Yau type and characteristic subvarieties over bounded symmetric domains, Math. Ann. 348, No. 1 (2012), 211-236.
		
		
		\bibitem[SZ16]{SZ}
		M. Sheng, K. Zuo, Deuring's mass formula for a Mumford family, Trans. Amer. Math. Soc. 368 (2016), no. 1, 169-207.
		
		
		\bibitem[Sim90]{S90}C. Simpson, Harmonic bundles on noncompact curves, J. Amer.
		Math. Soc. 3 (1990), no. 3, 713-770.
		
		\bibitem[Sim92]{Sim92}C. Simpson,  Higgs bundles and local systems,
		Inst. Hautes \'Etudes Sci. Publ. Math. No. 75 (1992), 5-95.
		
		\bibitem[T]{T} K. Takeuchi, A characterization of arithmetic Fuchsian groups, J. Math. Soc. Japan, Vol. 27, No. 4, 1975.
		
		\bibitem[TO]{TO}
		J. Tate, F. Oort, Group schemes of prime order. Ann. Sci. \'Ecole Norm. Sup. (4) 3 (1970), 1-21. 
		
		
		%
		
		
		
		
		
		
		
		%
		
		
		
		
		
			
			
		
		
		
		
		
		\bibitem[VZ]{VZ}
		E. Viehweg, K. Zuo, A characterization of certain Shimura curves in
		the moduli stack of abelian varieties, J. Differential Geom. 66 (2004), no. 2, 233-287.
		
		
		\bibitem[W]{Wort}D. Wortmann, The $\mu$-ordinary locus for Shimura varieties of Hodge type, arXiv preprint arXiv: 1310.6444(2013).
		
		
		\bibitem[Z]{Zhang}C. Zhang, Some remarks on Ekedahl-Oort stratifications, Taiwanese J. Math. 25 (2021), no. 4, 665-679. 
		
		
		
		
		
		
		
		
		
		
	\end{thebibliography}

\end{document}